\documentclass[a4paper,10pt]{article}
\usepackage{mathrsfs}
\usepackage{latexsym}
\usepackage{amsmath,amssymb}
\usepackage[pdftex,colorlinks]{hyperref}
\usepackage{amsthm}
\usepackage{amsfonts}
\usepackage[usenames]{color}
\usepackage{amssymb}
\usepackage{graphicx}
\usepackage{amsmath}
\usepackage{amsfonts}
\usepackage{amsthm}
\usepackage{mathrsfs}
\usepackage{dsfont}
\usepackage{indentfirst}

\newtheoremstyle{mythm}{1.5ex plus 1ex minus .2ex}{1.5ex plus 1ex
minus .2ex}{\kai}{\parindent}{\song\bfseries}{}{1em}{}
\numberwithin{equation}{section}

\newtheorem{theorem}{Theorem}[section]
\newtheorem{lemma}{Lemma}[section]

\newtheorem{remark}{Remark} [section]

\allowdisplaybreaks[4]
\begin{document}
\title{{\textbf{Existence of non-topological solutions for a skew-symmetric Chern-Simons system}}}

\author{Genggeng Huang\footnote{genggenghuang@sjtu.edu.cn}\quad and Chang-Shou Lin\footnote{cslin@math.ntu.edu.tw}}

\date{}
\maketitle
\begin{center}{\footnotesize{$^*$ Department of Mathematics, INS and MOE-LSC, Shanghai Jiao Tong University, Shanghai, China}}\\
{\footnotesize{$^\dag$  Taida Institute for Mathematical Sciences, Center for Advanced Study in Theoretical Science,\\
National Taiwan University, Taipei, 10617, Taiwan}}\end{center}

\begin{abstract}
We investigate the existence of non-topological solutions $(u_1,u_2)$ satisfying
 $$u_{i}(x)=-2\beta_i\ln|x|+O(1),\quad\text{as }|x|\rightarrow +\infty,$$ such that $\beta_i>1$ and
 $$(\beta_1-1)(\beta_2-1)>(N_1+1)(N_2+1),$$
 for a skew-symmetric Chern-Simons system.  By the bubbling analysis and the Leray-Schauder degree theory, we get the existence results except for a finite set of curves:
$$\frac{N_1}{\beta_1+N_1}+\frac{N_2}{\beta_2+N_2}=\frac{k-1}{k},k=2,\cdots,\max(N_1,N_2).$$
This generalizes a previous work  by Choe-Kim-Lin \cite{ChoeKimLin2011}.
\par {\bf 2010 Mathematics Subject Classification. 35J60,35J57
}
\end{abstract}

\section{Introduction}

In this paper, we study the nonlinear elliptic system:
\begin{equation}\label{CSS-1}
\begin{cases}
\Delta u_1+\lambda e^{u_2}(1-e^{u_1})=\displaystyle 4\pi\sum_{i=1}^{N_1}\delta_{p_{1i}},\\
\Delta u_2+\lambda e^{u_1}(1-e^{u_2})=\displaystyle 4\pi\sum_{i=1}^{N_2}\delta_{p_{2i}},
\end{cases}\text{ in } \mathbb R^2,
\end{equation}
where $\lambda>0$ and $\delta_p$ is the Dirac measure at $p$. System \eqref{CSS-1} arises in a relativistic Abelian Chern-Simons model involving two Higgs scalar fields and two gauge fields studied in \cite{Dziarmaga1994,KimLeeKoLeeMin1993}. The Chern-Simons action density for this physics model is defined on the $(2+1)$-dimensional Minkowski space $\mathbb R^{2,1}$ by
$$\mathcal L=-\frac 14 \kappa \epsilon^{rst}A_r^{(1)}F_{st}^{(2)}-\frac 14\kappa \epsilon^{rst}A_r^{(2)}F^{(1)}_{st}+\overline{D_r\phi_i}D^r\phi_i-V(\phi_1,\phi_2),$$
where $\phi_i,i=1,2$, are two complex scalar field representing two Higgs particles of charges $q_1$ and $q_2$. $A^{(i)}_r,i=1,2$, are two gauge fields with the induced electromagnetic fields $F_{rs}^{(i)}=\partial_r A_s^{(i)}-\partial_s A_r^{(i)}$, $r,s=0,1,2$, and $\kappa>0$ is the coupling constant, $D_r\phi^{(i)}=\partial_r\phi^{(i)}- \sqrt{-1} q_i A_r^{(i)}\phi^{(i)}$ is the covariant derivatives and
$$V(\phi_1,\phi_2)=\frac{q_1^2q_2^2}{\kappa^2}\left[|\phi_2|^2(|\phi_1|^2-c_1^2)^2+|\phi_1|^2(|\phi_2|^2-c_2^2)^2\right]$$
is the Higgs potential density. Even for a stationary solution, the Euler-Lagrangian equation for $\mathcal L$ is very complicated. In \cite{Dziarmaga1994,KimLeeKoLeeMin1993}, the authors considered the minimizer for the associated energy, which satisfies the following self-dual equation:
\begin{equation}
\label{phy1}
\begin{split}
& D_1\phi^{(i)}\pm D_2\phi^{(i)}=0\\
& F_{12}^{(i)}\pm \frac{2q_iq^2_{i+1}}{\kappa^2}|\phi^{(i+1)}|^2(|\phi^{(i)}|^2-c_i^2)=0,
\end{split}
\end{equation}
where for the simplicity of notation, we let $\phi^{(i)}\equiv \phi^{(j)},q_i\equiv q_j$ if $i\equiv j \ (mod \   2)$.
Let $u_i=\ln |\phi_i|^2$ and $p_{ij}$ are the zeros of $\phi_i$, $i=1,2$. Then system \eqref{phy1} can be transformed to \eqref{CSS-1}.
\par We note that from the second equation of \eqref{phy1}, both the quantities
\begin{equation}
\label{phy2}
\lambda\int_{\mathbb R^2} e^{u_2}(1-e^{u_1})dx \text{ and }\lambda\int_{\mathbb R^2} e^{u_1}(1-e^{u_2})dx
\end{equation}
represent the total magnetic flux for both components $u_i$, $i=1,2.$ Therefore, from the physic's point of view, it is important for us to find solutions with finite integral of \eqref{phy2} for both components. In literature, a solution $u=(u_1,u_2)$ is called topological if
$$u_i(x)\rightarrow 0\quad \text{as }|x|\rightarrow +\infty,i=1,2,$$
and is called non-topological if
$$u_i(x)\rightarrow -\infty\quad \text{as } |x|\rightarrow +\infty,i=1,2.$$
If $u$ is a topological solution, then $u_i(x)$ decays exponentially to zero, thus, both $e^{u_2}(1-e^{u_1})$ and $e^{u_1}(1-e^{u_2})\in L^1(\mathbb R^2)$.
If $u$ is a non-topological solution, then one of  $e^{u_2}(1-e^{u_1})$ and $e^{u_1}(1-e^{u_2})$ might be not in $L^1(\mathbb R^2)$. See \cite{ChernChenLin2010}.
\par For any configuration $\{p_{11},\cdots,p_{1N_1},p_{21},\cdots,p_{2N_2}\}$ in $\mathbb R^2$, the existence of a topological solution was obtained by Lin-Ponce-Yang \cite{LinPonceYang2007}. The proof was rather complicated if we compare it with the case of $SU(n)$ Chern-Simons system, as shown in \cite{Yang1997}, because system \eqref{CSS-1} is skew-symmetric, which is explained in the following.
\par To solve \eqref{CSS-1}, we might first solve the following regularized form:
\begin{equation}
\label{regular1}
\begin{cases}
\Delta u_1+\lambda e^{u_2}(1-e^{u_1})=\displaystyle \sum_{i=1}^{N_1}\frac{4\kappa}{(\kappa+|x-p_{1i}|^2)^2},\\
\Delta u_2+\lambda e^{u_1}(1-e^{u_2})=\displaystyle \sum_{i=1}^{N_2}\frac{4\kappa}{(\kappa+|x-p_{2i}|^2)^2},
\end{cases}
\end{equation}where $\kappa$ is a small positive constant. Then by letting $\kappa\rightarrow 0$, we could find a solution of \eqref{CSS-1}. To apply the variational method, we introduce the background functions,
$$u_0^\kappa(x)=\sum_{i=1}^{N_1}\ln\left(\frac{\kappa+|x-p_{1i}|^2}{1+|x-p_{1i}|^2}\right),
v_0^\kappa(x)=\sum_{i=1}^{N_2}\ln\left(\frac{\kappa+|x-p_{2i}|^2}{1+|x-p_{2i}|^2}\right).$$
Replacing $u_i$ by  $u_1-u_0^\kappa$ and  $u_2-v_0^\kappa$, the regularized form \eqref{regular1} becomes
\begin{equation}\label{regular2}
\begin{cases}
\Delta u_1+ \lambda e^{v_0^\kappa+u_2}(1-e^{u_0^\kappa+u_1})=h_1,\\
\Delta u_2+ \lambda e^{u_0^\kappa+u_1}(1-e^{v_0^\kappa+u_2})=h_2,
\end{cases}
\end{equation} where $h_1,h_2\in W^{1,2}$ do not depend on $\kappa >0$. By a direct computation, one can see that \eqref{regular2} is the Euler-Lagrange equation of the nonlinear functional:
\begin{equation}\label{regular3}
\begin{split}
I(u_1,u_2)=\int(\nabla u_1\cdot \nabla u_2+ \lambda e^{u_0^\kappa+v_0^\kappa+u_1+u_2}-\lambda e^{u_0^\kappa+u_1}-\lambda e^{v_0^\kappa+u_2}+h_2u_1+h_1u_2)dx
\end{split}
\end{equation}
 From \eqref{regular3}, the quadratic form $\nabla u_1\cdot \nabla u_2$ is not coercive, and this fact alone makes system \eqref{CSS-1} very difficult to study by variational method. Indeed, system \eqref{CSS-1} is a typical example of so called ``skew-symmetric" system. For its precise definition and recent development, we refer \cite{Yanagida2001,Yanagida2002}.
\par  In this paper, we want to find the non-topological solutions with finite magnetic flux, i.e., for a given $(\beta_1,\beta_2)$ with $\beta_i>1$, $i=1,2$, we want to find a solution $(u_1,u_2)$ of \eqref{CSS-1} such that
\begin{equation}\label{CSS-2}
u_i(x)=-2\beta_i\ln |x|+O(1), \quad i=1,2, \text{ near }\infty.
\end{equation}
When $u_1(x)=u_2(x)$, system \eqref{CSS-1} is reduced to the Abelian Chern-Simons equation,
\begin{equation}\label{s1}
\Delta u+ \lambda e^{u}(1-e^u)=4\pi\sum_{i=1}^N \delta_{p_i}.
\end{equation}
Equation  \eqref{s1} has been extensively studied in the whole $\mathbb R^2$ to search for topological solutions, non-topological solutions or in a flat torus to search for vortex condensates satisfying the periodic boundary condition, introduced by 't Hooft \cite{Hooft1979}. We refer reader to \cite{CaffarelliYang1995,ChanFuLin2002,Choe2007,ChoeKim2008,ChoeKimLin2011,Dunne1999,JaffeTaubes1980,LinYan2013,NolascoTarantello2000,SpruckYang1992,Tarantello1996} and references therein for recent development. In particular, in \cite{ChoeKimLin2011}, Choe-Kim-Lin proved the following existence theorem of non-topological solutions for \eqref{s1}.

\smallskip
\noindent
{\bf Theorem A.} \cite{ChoeKimLin2011} {\it Let $p_1,\cdots, p_N\in \mathbb R^2$ be given. For any number $\beta>N+2$ satisfying $\beta\notin\{\frac{N (k+1)}{k-1}|k=2,3,\cdots,N\}$, there exists a solution $u$ of \eqref{s1} satisfying
$$u(x)=-2\beta\ln |x|+O(1),\quad\text{near }\infty.$$
}
\par Our main result is to generalize Theorem A to system \eqref{CSS-1}. First, we consider the case of \eqref{CSS-1} when all the vortex points collapse into one point and \eqref{CSS-1} becomes:
\begin{equation}\label{CSS-6}
\begin{cases}
\Delta u_1+\lambda e^{u_2}(1-e^{u_1})=4\pi N_1\delta_{0},\\
\Delta u_2+\lambda e^{u_1}(1-e^{u_2})=\displaystyle 4\pi N_2\delta_{0},
\end{cases}\text{ in } \mathbb R^2,
\end{equation}
where $N_1,N_2\geq 0$ and $\delta_0$ is the Dirac measure at $0$. For \eqref{CSS-6}, we consider $u_i$ to be radially symmetric.
\begin{theorem}\label{thmradial}
Given $\beta_i>1$, $i=1,2$ satisfying
\begin{equation}\label{CSS-3}
(\beta_1-1)(\beta_2-1)>(N_1+1)(N_2+1),
\end{equation}
there exists a radial solution $u=(u_1(r),u_2(r))$ of  \eqref{CSS-6} such that \eqref{CSS-2} holds.
\end{theorem}
 Recently, Huang-Lin \cite{HuangLin2013} considered the existence of non-topological solutions of \eqref{CSS-6} with $N_1=N_2=0$. They showed that for any given pair $(\beta_1,\beta_2)$ with $1<\beta_i<\infty,i=1,2$ and $(\beta_1-1)(\beta_2-1)>1$, there exists a unique non-topological radial solution of \eqref{CSS-6} with $N_1=N_2=0$ such that
\begin{equation*}
u_i(x)=-2\beta_i\ln |x|+O(1), \quad i=1,2, \text{ near }\infty.
\end{equation*}
The proof of their result is based on non-degeneracy of linearized equations. Using this result, we may prove Theorem \ref{thmradial} by deforming system \eqref{CSS-6} by
\begin{equation}
\label{CSS-7}
\begin{cases}
\Delta u_1+\lambda e^{u_2}(1-e^{u_1})=4\pi \epsilon N_1\delta_{0},\\
\Delta u_2+\lambda e^{u_1}(1-e^{u_2})=\displaystyle 4\pi \epsilon N_2\delta_{0},
\end{cases}\text{ in } \mathbb R^2
\end{equation}
for $\epsilon \in[0,1]$. Suitably applying the Pohozaev's identity, we can show that for any solution $(u_1,u_2)$ of \eqref{CSS-7} for fixed $\beta_i$, $\beta_i>1$ such that \eqref{CSS-3} holds, $u_i$ is uniformly bounded in some function space, which the classical Leray-Schauder degree(see \cite{Nirenberg2001}) can be applied. Therefore, the degree for \eqref{CSS-7} is invariant under the deformation. For $\epsilon=0$, the degree is equal to $-1$ due to the result of \cite{HuangLin2013}. Thus, For \eqref{CSS-6} i.e, $\epsilon=1$, the degree is also equal to $-1$, and then the existence follows immediately. The complete proof will be given in Section 2.
\par For any configuration $\{p_{11},\cdots,p_{1N_1},p_{21},\cdots,p_{2N_2}\}$ in $\mathbb R^2$, we have the following theorem which extends Theorem A to system \eqref{CSS-1}.
\begin{theorem}\label{mainthm1}Let $p_{11},\cdots,p_{1N_1},p_{21},\cdots,p_{2N_2}$ be given. For any $(\beta_1,\beta_2)$  satisfying \eqref{CSS-3}
and
\begin{equation}
\label{CSS-4}\frac{N_1}{\beta_1+N_1}+\frac{N_2}{\beta_2+N_2}\notin\{\frac{k-1}{k}|k=2,\cdots,\max(N_1,N_2)\}.
\end{equation}
Then there exists a solution $(u_1,u_2)$ solves \eqref{CSS-1}, \eqref{CSS-2}.
\end{theorem}
It is easy to see that if we take $N_1=N_2$, $\beta_1=\beta_2$ and $\{p_{11},\cdots,p_{1N_1}\}=\{p_{21},\cdots,p_{2N_2}\}$, then we can prove $u_1=u_2$. In this case, Theorem \ref{mainthm1} is the same as Theorem A.
\par Our proof uses the same strategy as Theorem \ref{thmradial}, but the proof is more involved. The basic observation is that sometimes collapsing vortex points might not cause bubbling for system \eqref{CSS-1}. Hence, as in the proof of Theorem \ref{thmradial},  we want to establish  a priori estimates of the following deformed system,
\begin{equation}\label{CSS-5}
\begin{cases}
\Delta u_1+\lambda e^{u_2}(1-e^{u_1})=\displaystyle 4\pi\sum_{i=1}^{N_1}\delta_{\epsilon p_{1i}},\\
\Delta u_2+\lambda e^{u_1}(1-e^{u_2})=\displaystyle 4\pi\sum_{i=1}^{N_2}\delta_{\epsilon p_{2i}},
\end{cases}\text{ in } \mathbb R^2
\end{equation}
where $\epsilon\in[0,1]$. For $\epsilon=1$, it is just the same system as \eqref{CSS-1}, and $\epsilon=0$, \eqref{CSS-5} is reduced to \eqref{CSS-6}.
This a priori estimates could be obtained through the so-called ``bubbling analysis". This kind of technique began with the celebrated work of Brezis-Merle \cite{BrezisMerle1991} on the scalar nonlinear equation with exponential nonlinearity and has been developped into a powerful method through the works by Li-Shafrir\cite{LiShafrir1994}, Bartolucci-Chen-Lin-Tarantello\cite{BartolucciChenLinTarantello2004} and Chen-Lin\cite{ChenLin2002}.
For our situation, we have to extend \cite{BrezisMerle1991} to the following system:
\begin{equation}\label{s3}
\begin{cases}
\Delta u_{1n}+V_{1n}e^{u_{2n}}=0,\\
\Delta u_{2n}+V_{2n}e^{u_{1n}}=0,
\end{cases}\text{ in }\Omega\subset\subset \mathbb R^2.
\end{equation}
There is an interesting feature for bubbling solutions of \eqref{s3}: there exist a sequence of $x_{1n}\rightarrow \bar x$ such that $u_{1n}(x_{1n})\rightarrow +\infty$ iff there exist a sequence of $x_{2n}\rightarrow \bar x$ such that $u_{2n}(x_{2n})\rightarrow +\infty$. Thus, one consequence of our analysis is that both  $V_{1n}e^{u_{2n}}$ and $V_{2n}e^{u_{1n}}$ converges to $\sum_{q\in S}M_q\delta_q$ and $\sum_{q\in S}N_q\delta_q$ in measure, which implies the concentration phenomenon also occurs for system \eqref{CSS-1}.
\par The paper is organized as follows. In Section 2, we will calculate the Leray-Schauder degree of the radially symmetric solutions for \eqref{CSS-7} and prove Theorem \ref{thmradial}. We will generalize Brezis-Merle's alternative for PDE system \eqref{s3}  in Section 3. In Section 4, we will establish a Pohozaev's identity and then obtain the a priori estimates for \eqref{CSS-5} by bubbling analysis. Theorem \ref{mainthm1} will be proved in Section 5 by Leray-Schauder degree theory.

\section{Radially symmetric solutions}

In this section, we will prove Theorem \ref{thmradial} by establishing a priori estimates of the radial solutions and using Leray-Schauder degree theory.
We consider the following system:
\begin{equation}
\label{CSSS-7}
\begin{cases}
\Delta u_1+e^{u_2}(1-e^{u_1})=4\pi \epsilon N_1\delta_{0},\\
\Delta u_2+ e^{u_1}(1-e^{u_2})=\displaystyle 4\pi \epsilon N_2\delta_{0},
\end{cases}\text{ in } \mathbb R^2
\end{equation}
for $\epsilon \in[0,1]$ where
\begin{equation}\label{CSSS-2}
u_i(x)=-2\beta_i\ln |x|+O(1), \quad\beta_i>1,\quad i=1,2, \text{ near }\infty,
\end{equation}
and
\begin{equation}\label{CSSS-3}
(\beta_1-1)(\beta_2-1)>(N_1+1)(N_2+1).
\end{equation}

We begin our proof with the following Lemma.
\begin{lemma}
\label{lemdeform1}
Suppose $(u_1(r),u_2(r))$ solves  \eqref{CSSS-7}, \eqref{CSSS-2}, \eqref{CSSS-3}. Then
\begin{equation}
\label{deform3}
|r u_{i}'(r)|\leq C,\quad \text{for } r>0
\end{equation}
and
\begin{equation}
\label{deform4}
u_{i}(r)\leq -2\ln r+C,\quad \text{for } r\geq 1
\end{equation}
where $C$ is a constant depending only on $\beta_i,N_i$.
\end{lemma}
\begin{proof}
In fact, by directly integrating, we will have
$$ru_1'(r)-2\epsilon N_1=-\int_0^r e^{u_2}(1-e^{u_1}) s ds,\quad ru_2'(r)-2\epsilon N_2=-\int_0^r e^{u_1}(1-e^{u_2}) s ds.$$
\eqref{deform3} follows easily from these two equalities and Pohozaev's identity in Lemma \ref{lemPI}. For $s\geq r\geq 1$, from \eqref{deform3}, one gets
\begin{equation*}
u_i(s)=u_i(r)+\int_r^s u'_i(\lambda )d\lambda\geq u_i(r)-C\ln \frac sr.
\end{equation*}
By Lemma \ref{lemPI}, we have
\begin{equation*}
C'\geq \int_{r}^\infty s e^{u_i(s)}ds\geq \int_r^\infty s e^{u_i(r)} \left(\frac rs\right)^{C} ds=r^2e^{u_i(r)}\int_1^\infty t^{1-C}dt\geq cr^2e^{u_i(r)}.
\end{equation*}
This proves \eqref{deform4}.
\end{proof}
\begin{lemma}\label{lemdeform2}
Suppose $(u_1(r),u_2(r))$ solves  \eqref{CSSS-7}, \eqref{CSSS-2}, \eqref{CSSS-3}. Then
$$|u_1|_{L^\infty(K)}+|u_2|_{L^\infty(K)}\leq C_K,\quad  \forall K\subset\subset \mathbb R^2\backslash\{0\}.$$
\end{lemma}
\begin{proof}
If not, one may assume that there exists a sequence of positive numbers $r_n\rightarrow r^*>0$ and $\epsilon_n\rightarrow \epsilon^*$ such that $u_{1n}(r^*)\rightarrow -\infty$ as $n\rightarrow \infty$.
From \eqref{deform4}, one gets
$u_{1n}(r)\rightarrow -\infty$ as $r\rightarrow +\infty$ uniformly for all $n$. Hence, we have $\max_{\mathbb R^2}u_{1n}\rightarrow -\infty$ as $n\rightarrow \infty$. By Pohozaev's identity in Lemma \ref{lemPI}, we have
\begin{equation*}
o(1)=\max_{\mathbb R^2}e^{u_{1n}}\int_{\mathbb R^2}e^{u_{2n}}dx\geq \int_{\mathbb R^2} e^{u_{1n}+u_{2n}}dx=4\pi((\beta_1-1)(\beta_2-1)-(\epsilon_nN_1+1)(\epsilon_nN_2+1))\geq c_0>0.
\end{equation*}
This yields a contradiction.
\end{proof}

Set $h_{i\epsilon}=2\epsilon N_i\ln r-(\beta_i+\epsilon N_i)\ln(1+r^2), $ $i=1,2$.
\begin{lemma}\label{thmdeform1}
Suppose $(u_{1}(r),u_2(r))$ solves \eqref{CSSS-7}, \eqref{CSSS-2}, \eqref{CSSS-3}.
Then $$|u_1-h_{1\epsilon}|_{L^\infty(\mathbb R^2)}+|u_2-h_{2\epsilon}|_{L^\infty(\mathbb R^2)}\leq C.$$
\end{lemma}
\begin{proof}
Let $v_{1}, v_{2}$ be defined as $v_i=u_i-h_{i\epsilon}$, $i=1,2$.
By a direct computation, $v_{i}$ should satisfy
\begin{equation}\label{radial1}
\begin{cases}
\displaystyle \Delta v_1 +e^{u_2}(1-e^{u_1})=\frac{4(\epsilon N_1+\beta_1)}{(1+|x|^2)^2},\\
\displaystyle \Delta v_2 +e^{u_1}(1-e^{u_2})=\frac{4(\epsilon N_2+\beta_2)}{(1+|x|^2)^2},
\end{cases}\quad\text{in }\mathbb R^2.
\end{equation}
Since $u_1,u_2<0$ and $u_1,u_2$ are uniformly bounded in $L^\infty(\partial B_1)$, we can take $\eta(x)=C|x|^2$ as a barrier function for some suitable large constant $C$. This proves that $v_i\in L^\infty(B_1)$, $i=1,2.$ In order to show $v_i\in L^\infty(|x|\geq 1)$, one should take Kelvin transformation as follows:
$$\xi_i(x)=u_i(x/|x|^2)-2\beta_i\ln |x|,\quad |x|\leq 1,\quad i=1,2.$$
Then $\xi_i$ satisfies
\begin{equation*}
\begin{split}
\Delta \xi_{1}+|x|^{2\beta_2-4}e^{\xi_{2}}(1-|x|^{2\beta_1}e^{\xi_{1}})=0,\\
\Delta \xi_{2}+|x|^{2\beta_1-4}e^{\xi_{1}}(1-|x|^{2\beta_2}e^{\xi_{2}})=0,
\end{split}\quad\quad\text{ in } |x|\leq 1.
\end{equation*}
We need to show $\xi_i$ is bounded from above. If not, we may assume $\max_{B_1}\xi_{1n}=\xi_{1n}(r_n)\rightarrow +\infty$, $r_n\rightarrow 0,\epsilon_n\rightarrow \epsilon^*$ as $n\rightarrow \infty$.    Set
$$s_n=\min\{e^{-\frac{\lambda_{1n}}{2\beta_1-2}},e^{-\frac{\lambda_{2n}}{2\beta_2-2}}\}=e^{-\frac{\lambda_{1n}}{2\beta_1-2}}\rightarrow 0,$$
where $\lambda_{in}=\max_{B_1}\xi_{in}$, $i=1,2$.
 In fact,  By \eqref{deform4}, we will have $\xi_{1n}(r_n)+2(\beta_1-1)\ln r_n\leq C$ which means $\frac{r_n}{s_n}\leq C$.

Set
$$\bar \xi_{in}(x)=\xi_{in}(s_n  x)-\lambda_{in},i=1,2.$$
Then we have $\bar \xi_{in}\leq 0,i=1,2,$ and
\begin{equation*}
\begin{cases}
\displaystyle -\Delta \bar \xi_{1n}=e^{\lambda_{2n}+2(\beta_2-1)\ln s_n}|x|^{2\beta_2-4}e^{\bar \xi_{2n}}(1-r_n^2s_n^2|x|^{2\beta_1}e^{\bar \xi_{1n}}),\\
\displaystyle-\Delta \bar \xi_{2n}=|x|^{2\beta_1-4}e^{\bar \xi_{1n}}(1-r_n^2s_n^{2\beta_2}e^{\lambda_{2n}}|x|^{2\beta_2}e^{\bar \xi_{2n}}),
\end{cases}\text{ in } |x|\leq 1/s_n.
\end{equation*}Also we have $\bar \xi_{1n}(x)\leq \bar\xi_{1n}(\frac{r_{n}}{s_n})=0$. Then by standard elliptic estimates and $\lambda_{2n}+2(\beta_2-1)\ln s_n\leq 0$, we have $\bar \xi_{1n}\rightarrow \bar\xi_1$ in $W^{2,\gamma}_{loc}(\mathbb R^2)$ for some $\gamma>1$. As for $\bar \xi_{2n}$, by Harnack inequality, one gets either $\bar\xi_{2n}\rightarrow -\infty$ locally uniformly in $\mathbb R^2$ or $\bar\xi_{2n}\rightarrow \bar \xi_2$ in $W^{2,\gamma}_{loc}(\mathbb R^2)$. By $\int_{\mathbb R^2}|x|^{2\beta_1-4}e^{\bar \xi_1}dx<+\infty$, we can exclude the previous case. In fact, by the same arguments, we also have $\lambda_{2n}+2(\beta_2-1)\ln s_n\geq -C$, otherwise,
\begin{equation*}
\Delta \bar \xi_1=0,\bar \xi_1(x_0)=0, x_0=\lim_{n\rightarrow \infty}\frac{r_{n}}{s_n}\Rightarrow \bar \xi_1\equiv 0
\end{equation*}which contradicts to $\int_{\mathbb R^2}|x|^{2\beta_1-4}e^{\bar \xi_1}dx<+\infty$ by Fatou's lemma. And $\bar \xi_1,\bar \xi_2$ satisfy
\begin{equation*}
\begin{cases}
&-\Delta \bar\xi_1=c_0|x|^{2\beta_2-4}e^{\bar \xi_2},\\
&-\Delta \bar \xi_2=|x|^{2\beta_1-4}e^{\bar \xi_1},
\end{cases}\text{ in }\mathbb R^2
\end{equation*}with $|x|^{2\beta_1-4}e^{\bar\xi_1},|x|^{2\beta_2-4}e^{\bar\xi_2}\in L^1(\mathbb R^2)$ for some constant $0<c_0\leq 1$. Set
\begin{equation*}
A_1=\frac1{2\pi}\int_{\mathbb R^2}c_0|x|^{2\beta_2-4}e^{\bar \xi_2}dx,A_2=\frac1{2\pi}\int_{\mathbb R^2}|x|^{2\beta_1-4}e^{\bar \xi_1}dx.
\end{equation*}By Pohozaev's identity and repeating the arguments in Lemma \ref{lemPI}, one gets
$$A_1A_2=2(\beta_2-1)A_1+2(\beta_1-1)A_2.$$Noting $2(\beta_1+\epsilon^*N_1)\geq A_1,2(\beta_2+\epsilon^*N_2)\geq A_2$, we have

\begin{equation*}
\frac{\beta_1-1}{\beta_1+\epsilon^*N_1}+\frac{\beta_2-1}{\beta_2+\epsilon^*N_2}\leq 1, \quad \text{ i.e., } \beta_1\beta_2-\beta_1-\beta_2\leq \epsilon^*(\epsilon^*N_1N_2+N_1+N_2)
\end{equation*}
which contradicts to \eqref{CSSS-3}. This proves the upper bound of $\xi_i$.  By Harnack inequality and $\xi_i\in L^\infty(\partial B_1)$, we get $\xi_i$ is bounded in $L^\infty(B_1)$. This proves Lemma \ref{thmdeform1}.
\end{proof}
By Lemma \ref{thmdeform1}, we now can calculate the correponding Leray-Schauder degree of \eqref{radial1}.
Introduce the following Hilbert space for $\beta=\min(\beta_1,\beta_2,2)>1$
$$\mathcal{D}=\{v:\mathbb R^2\rightarrow \mathbb R\ |\   |v|_{\mathcal{D}}^2=\int_{\mathbb R^2}|\nabla v|^2 dx+\int_{\mathbb R^2}\frac{v^2}{(1+|x|^2)^{\beta}}dx<+\infty\}.$$
We denote the radial subspace of  $\mathcal D $ by $\mathcal D_r$ and $\mathcal D_r^2=\mathcal D_r\times \mathcal D_r$. we define a map:
 $$G(\epsilon,v_1,v_2)=(G_1(\epsilon,v_1,v_2),G_2(\epsilon,v_1,v_2)): \mathcal D_r^2\rightarrow \mathcal D_r^2$$
for $t\in [0,1]$ by
\begin{equation}\label{degreemap1}
\begin{split}
& G_{1}(\epsilon,v_1,v_2)=(-\Delta+\sigma)^{-1}[e^{v_2+h_{2\epsilon}}(1-e^{v_1+h_{1\epsilon}})+\sigma v_1-g_{1\epsilon}],\\
& G_{2}(\epsilon,v_1,v_2)=(-\Delta+\sigma)^{-1}[e^{v_1+h_{1\epsilon}}(1-e^{v_2+h_{2\epsilon}})+\sigma v_2-g_{2\epsilon}],
\end{split}
\end{equation}where $\sigma=\frac {1}{(1+|x|^2)^\beta}$ and $g_{i\epsilon}=\frac{4(\epsilon N_i+\beta_i)}{(1+|x|^2)^2}$, $i=1,2$. For $\epsilon\in [0,1]$, $G(\epsilon,\cdot,\cdot)$ is a continuous compact operator from $\mathcal D_r^2$ to $\mathcal D_r^2$.
\par \textbf{The proof for Theorem \ref{thmradial}:}
\par
Under the assumptions of Lemma \ref{thmdeform1}, we have $|v_1|_{\mathcal D_r}+|v_2|_{\mathcal D_r}\leq C$ uniformly with respect to $\epsilon\in[0,1]$.
One can use Lemma \ref{thmdeform1} and integrate by parts to get it. We omit the details here.  Therefore, we can choose a number $R>0$ independent of $\epsilon$ such that if $I-G(\epsilon,v_1,v_2)=0$ then $|v_1|_{\mathcal D_r}+|v_2|_{\mathcal D_r}<R$. Set $\Omega_R=\{(v_1,v_2)\in\mathcal D_r\times \mathcal D_r\big{|}|v_1|_{\mathcal D_r}+|v_2|_{\mathcal D_r}< R\}$. Therefore the degree $\deg(I-G(\epsilon,\cdot,\cdot),\Omega_R,0)$ in $\mathcal D_r\times\mathcal D_r$ is well defined for $\epsilon\in[0,1]$. Also, $I-G(\epsilon,\cdot,\cdot)$ defines a good homotopy.
We have $\deg(I-G(1,v_1,v_2),\Omega_R,0)=\deg(I-G(0,v_1,v_2),\Omega_R,0)$. But $I-G(0,v_1,v_2)=0$ is equivalent to
\begin{equation*}
\begin{cases}
\Delta v_1+e^{v_2}(1-e^{v_1})=0,\\
\Delta v_2+e^{v_1}(1-e^{v_2})=0,
\end{cases}\quad \text{in }\mathbb R^2.
\end{equation*}
It is already known that for radially symmetric solutions of $I-G(0,v_1,v_2)=0$, one has $\deg(I-G(0,v_1,v_2),\Omega_R,0)=-1$, see \cite{HuangLin2013}. This proves Theorem \ref{thmradial}.

\section{Generalization of Brezis-Merle's alternative}
Before proving Theorem \ref{mainthm1}, we will first generalize the Brezis-Merle's alternative for PDE systems as follows in this section .
Consider the following system:
\begin{equation}\label{BM-1}
\begin{cases}
\Delta u_{1n}+V_{1n}e^{u_{2n}}=0,\\
\Delta u_{2n}+V_{2n}e^{u_{1n}}=0,
\end{cases}\text{ in }\Omega\subset\subset \mathbb R^2.
\end{equation}
\begin{lemma}\label{lemBM1}
Assume $(u_{1n},u_{2n})$ is a sequence of solutions of \eqref{BM-1} satisfying
\begin{eqnarray*}
&&|V_{1n}|_{L^{\infty}(\Omega)}+|V_{2n}|_{L^{\infty}(\Omega)}\leq C_1,|e^{u_{1n}}|_{L^{1}(\Omega)}+|e^{u_{2n}}|_{L^{1}(\Omega)}\leq C_2,\\
&& \text{either }\int_{\Omega}|V_{1n}|e^{u_{2n}}dx\leq \epsilon_1<{4\pi},\text{ or }\int_{\Omega}|V_{2n}|e^{u_{1n}}dx\leq \epsilon_2<{4\pi}.
\end{eqnarray*}Then $u^+_{1n},u^+_{2n}$ are uniformly bounded in $L^\infty_{loc}(\Omega)$.
\end{lemma}
\begin{proof}
Split $u_{kn}=\bar u_{kn}+\tilde u_{kn},k=1,2$ such that
\begin{equation}
\begin{cases}
\Delta \bar u_{1n}+V_{1n}e^{u_{2n}}=0,\\
\Delta \bar u_{2n}+V_{2n}e^{u_{1n}}=0,\quad \text{in }\Omega\\
\bar u_{1n}=\bar u_{2n}=0\text{ on }\partial\Omega.
\end{cases}
\end{equation}Now if $\int_{\Omega}|V_{1n}|e^{u_{2n}}dx\leq \epsilon_1<{4\pi}$, by Brezis-Merle's inequality, we know that
$\int_{\Omega}e^{(1+\delta)|\bar u_{1n}|}dx\leq C$ for some small $\delta>0$ and also $\bar u_{1n}, \bar u_{2n}$ are uniformly bounded in $ L^1(\Omega)$.
 Noting that $\tilde u_{kn}$ is harmonic in $\Omega$ and $$|\tilde u^+_{kn}|_{L^1(\Omega)}\leq |u^+_{kn}|_{L^1(\Omega)}+|\bar u_{kn}|_{L^1(\Omega)}\leq C,$$ we will have $\tilde u^+_{kn}$ is  bounded in  $L^\infty_{loc}(\Omega)$ by the mean value property of harmonic functions. This means that $e^{u_{1n}}$ is bounded in $L^{1+\delta}_{loc}(\Omega)$. Applying the standard elliptic estimates to $\bar u_{2n}$ yields that $\bar u_{2n}$ is bounded in $W^{2,\gamma}_{loc}(\Omega)$ for some $\gamma>1$, i.e. $\bar u_{2n}$ is  bounded in $L^{\infty}_{loc}(\Omega)$. Taking this estimate back to the equation of $\bar u_{1n}$, we must have $\bar u_{1n}$ is  bounded in $L^{\infty}_{loc}(\Omega)$. This proves the present lemma.
\end{proof}
Define the blow-up set $S$ as follows.
\begin{eqnarray*}S=\{x\in\Omega; &&\text{ there exist two sequences $x_{1n}$, $x_{2n}$ such that }x_{1n},x_{2n}\rightarrow x,\\&&u_{1n}(x_{1n}),u_{2n}(x_{2n})\rightarrow +\infty\}.\end{eqnarray*}
\begin{theorem}\label{thmBM1}
Assume $(u_{1n},u_{2n})$ is a sequence of solutions of \eqref{BM-1} with
$$|V_{1n}|_{L^{\infty}(\Omega)}+|V_{2n}|_{L^{\infty}(\Omega)}\leq C_1,|e^{u_{1n}}|_{L^{1}(\Omega)}+|e^{u_{2n}}|_{L^{1}(\Omega)}\leq C_2$$ Then there exists a subsequence of $(u_{1n},u_{2n})$ satisfying that
\begin{itemize}
\item[(i).]  $S=\emptyset$. Then for each component $u_{kn}$, it either is  bounded in $L^\infty_{loc}(\Omega)$ or locally uniformly converges to $-\infty$.
\item[(ii).] $S\neq \emptyset$. Then $\forall x\in S$, $\exists x_{1n},x_{2n}\rightarrow x,u_{1n}(x_{1n}),u_{2n}(x_{2n})\rightarrow +\infty$. Moreover, $\forall K\subset\subset\Omega\backslash S$, $u_{1n}(x_{1n}),u_{2n}(x_{2n})\rightarrow -\infty$ uniformly on $K$ and $$
    V_{1n}e^{u_{2n}}\rightarrow \sum_{r\in S}a_{1r}\delta_r,V_{2n}e^{u_{1n}}\rightarrow \sum_{r\in S}a_{2r}\delta_r,a_{1r},a_{2r}\geq {4\pi}$$ in the sense of measure and $S$ is a finite set.
\end{itemize}
\end{theorem}
\begin{proof}
Set $ V_{1n}e^{u_{2n}}\rightarrow \mu_1, V_{2n}e^{u_{1n}}\rightarrow  \mu_2$ in  measure in $\Omega$. Define
$$\Sigma=\{x\in\Omega|\mu_1(\{x\})\geq {4\pi},\mu_2(\{x\})\geq {4\pi}\}.$$
\par Step 1. $S=\Sigma$. First $S\subset \Sigma$. If $x_0\notin\Sigma$, without loss of generality, we may choose $\delta_0$ small enough such that $\int_{B_{\delta_0(x_0)}}V_{1n}e^{u_{2n}}<{4\pi}$. Applying Lemma \ref{lemBM1} in $B_{\delta_0}(x_0)$, one can get $$
u_{1n}^+,u_{2n}^+ \text{ are bounded in } L^\infty_{loc}(B_{\delta_0(x_0)})$$ which means $x_0\notin S$. This proves $S\subset \Sigma$. Picking up $x_0\in \Sigma$, we claim that
$$\forall R>0,\lim_{n\rightarrow \infty}\inf|u^+_{1n}|_{L^\infty(B_R(x_0))}\rightarrow +\infty,\lim_{n\rightarrow \infty}\inf|u^+_{2n}|_{L^\infty(B_R(x_0))}\rightarrow +\infty.$$ Suppose not, we may assume that $$|u^+_{1n}|_{L^\infty(B_{R_0}(x_0))}\leq C,\quad\text{for some }R_0,\quad\text{uniformly for some constant }C.$$
Then by H\"older's inequality, we can take $R<R_0$ small enough such that $\int_{B_R(x_0)}V_{2n}e^{u_{1n}}dx<{4\pi}$. This contradicts with $x_0\in \Sigma$ and proves the claim. Set
$$u_{1n}(x_{1n})=\max_{B_R(x_0)}u_{1n},u_{2n}(x_{2n})=\max_{B_R(x_0)}u_{2n}, B_R(x_0)\cap \Sigma=\{x_0\}.$$ Then $u_{1n}(x_{1n}),u_{2n}(x_{2n})\rightarrow +\infty$, $x_{1n}\rightarrow x_1,x_{2n}\rightarrow x_2$. We need to show $x_1=x_2=x_0$. If not, we may assume $x_1\neq x_0$. By the choice of $B_R(x_0)$, we must have $x_1\notin \Sigma$ which means $u^+_{1n},u^+_{2n}$ are bounded in $L^\infty(B_{\delta}(x_1))$ for some small $\delta>0$. This yields a contradiction and $x_0\in S$. This proves $S=\Sigma$ and $S$ is a finite set.
\par Step 2. $S=\emptyset$ means (i) holds. First if $S=\emptyset$, we will have $u_{1n}^+,u_{2n}^+\in L^\infty_{loc}(\Omega)$. If not, we may assume that $u_{1n}(x_{1n})\rightarrow +\infty$, $x_{1n}\rightarrow x_1\in \Omega$. We claim that:
$$\exists x_{2n}\rightarrow x_1\text{ such that }u_{2n}(x_{2n})\rightarrow +\infty.$$ If not, there exists $\delta_0>0$ such that $u^+_{2n}$ is bounded in $ L^\infty(B_{\delta_0}(x_1))$. Repeating the proof in Lemma \ref{lemBM1}, one can get $u_{1n}^+$ is bounded in $L^\infty_{loc}(B_{\delta_0}(x_1))$ contradicts to our assumption. This implies $x_1\in S$ contradicts to our assumption again. Hence we have $u_{1n}^+,u_{2n}^+$ are bounded in $L^\infty_{loc}(\Omega)$. Then we can apply Harnack inequality to \eqref{BM-1} to get (i) holds.
\par Step 3. $S\neq \emptyset$ implies (ii) holds. By Lemma \ref{lemBM1}, we know that $u_{1n}^+,u^+_{2n}$ are bounded in $L^\infty_{loc}(\Omega\backslash S)$, which means $V_{1n}e^{u_{2n}}\in L^{\infty}_{loc}(\Omega\backslash S),V_{2n}e^{u_{1n}}\in L^{\infty}_{loc}(\Omega\backslash S)$.  This implies $\mu_1,\mu_2$ are bounded measure on $\Omega$ with $\mu_1\in L^{\infty}_{loc}(\Omega\backslash S),\mu_2\in L^{\infty}_{loc}(\Omega\backslash S)$. With $\bar u_{kn},\tilde u_{kn},k=1,2$ defined as in Lemma \ref{lemBM1}, we have $\bar u_{1n},\bar u_{2n}\rightarrow \bar u_1,\bar u_2$ locally uniformly in $\Omega\backslash S$. Also by mean value property, we have $\tilde u_{kn}^+\leq C$. Applying the Harnack inequality yields
\begin{itemize}
\item[(a).] At least one component of $(\tilde u_{1n},\tilde u_{2n})$ is bounded in $L^\infty_{loc}(\Omega\backslash S)$.
\item[(b).] $\tilde u_{1n},\tilde u_{2n}\rightarrow -\infty$ locally uniformly in $\Omega\backslash S$.
\end{itemize}
We exclude situation (a) as follows. If (a) happens, we may assume $\tilde u_{1n}\in L^\infty_{loc}(\Omega\backslash S)$. Consider $x_0\in S$. Then for small $R\leq R_0$, $\tilde u_{1n}\in L^\infty(\partial B_R(x_0))$, $| u_{1n}|_{L^\infty(\partial B_R(x_0))}\leq C$. Consider the following boundary value problem:\begin{equation*}
\begin{cases}
-\Delta h_{1n}=V_{1n}e^{u_{2n}},\quad \text{in }B_R(x_0),\\
h_{1n}=-C,\quad \text{in }\partial B_R(x_0).
\end{cases}
\end{equation*}Then by the maximal principle, we have $$u_{1n}\geq h_{1n},\quad \text{in }B_R(x_0).$$ In particular $\int_{B_R(x_0)}e^{ h_{1n}}dx\leq \int_{B_R(x_0)}e^{u_{1n}}dx\leq C<\infty$. On the other hand, we have $h_{1n}\rightarrow h_1 \in W^{2,q}_{loc}(\bar B_R(x_0)\backslash \{0\}),\forall q<\infty$ with $h_1$ solves
\begin{equation*}
\begin{cases}
-\Delta h_{1}=\mu_1,\quad \text{in }B_R(x_0),\\
h_{1}=-C,\quad \text{on }\partial B_R(x_0).
\end{cases}
\end{equation*} As $x_0\in S$, we have $\mu_1\{x_0\}\geq {4\pi}$ which implies $\mu_1\geq {4\pi}\delta_{x_0}$. One gets that $$h_1(x)\geq -2\ln|x-x_0|+O(1). $$ Then $\int_{B_R(x_0)}e^{ h_1}dx=\infty$ yields a contradiction. Thus we must have situation (b) happens. This ends the proof of our theorem.
\end{proof}
\begin{remark}
When we apply Theorem \ref{thmBM1} to system \eqref{CSS-5}, we will obtain  $e^{u_{1n}}, e^{u_{2n}}$ are uniformly bounded in $L^1(\mathbb R^2)$ by Pohozaev's identity even though we only have $V_{1n}e^{u_{2n}},V_{2n}e^{u_{1n}}$ are uniformly bounded in $L^1(\mathbb R^2)$. This will be found in Lemma \ref{lemPI}.
\end{remark}

\section{A priori estimates}
In this section, we will obtain the a priori estimates for the solutions of the following problem
\begin{equation}\label{CS-5}
\begin{cases}
\Delta u_1+ e^{u_2}(1-e^{u_1})=\displaystyle 4\pi\sum_{i=1}^{N_1}\delta_{\epsilon p_{1i}},\\
\Delta u_2+ e^{u_1}(1-e^{u_2})=\displaystyle 4\pi\sum_{i=1}^{N_2}\delta_{\epsilon p_{2i}},
\end{cases}\text{ in } \mathbb R^2
\end{equation}
where $\epsilon\in[0,1]$ with
\begin{equation}\label{CS-2}
u_i(x)=-2\beta_i\ln |x|+O(1), \quad \beta_i>1, \quad i=1,2, \text{ near }\infty.
\end{equation}
and
\begin{equation}\label{CS-3}
(\beta_1-1)(\beta_2-1)>(N_1+1)(N_2+1)
\end{equation}
 by blow-up analysis.
 In what follows we always assume $\max(N_1,N_2)\geq 1$. Otherwise, this is the 0-vortex case which has been discussed in \cite{HuangLin2013}.
Now suppose $(u_1,u_2)$ is a solution of system \eqref{CS-5}, \eqref{CS-2}, \eqref{CS-3}. Then  applying the maximum principle, we have $u_1,u_2<0$, $\forall x\in \mathbb R^2$. Write
$$u_1(x)=v_1(x)+f_{1\epsilon}(x),u_2(x)=v_2(x)+f_{2\epsilon}(x)$$ where $f_{1\epsilon}(x)=2\displaystyle \sum_{i=1}^{N_1}\ln|x-\epsilon p_{1i}|,f_{2\epsilon}(x)=2\displaystyle \sum_{i=1}^{N_2}\ln|x-\epsilon p_{2i}|$. At first, we shall establish the following Pohozaev's identities.
\begin{lemma}
\label{lemPI}Let $(u_1,u_2)$ be a solution of \eqref{CS-5}, \eqref{CS-2}, \eqref{CS-3}. Then $(u_1,u_2)$ satisfies
\begin{equation*}
\begin{split}
\int_{\mathbb R^2} e^{u_1} dx=&4\pi(\beta_1\beta_2-N_1N_2-\beta_1-N_1)-2\pi\left(\sum_{i=1}^{N_1}\epsilon p_{1i}\cdot\nabla v_2(\epsilon p_{1i})+\sum_{i=1}^{N_2}\epsilon p_{2i}\cdot\nabla v_1(\epsilon p_{2i})\right)\\
\int_{\mathbb R^2} e^{u_2} dx=&4\pi(\beta_1\beta_2-N_1N_2-\beta_2-N_2)-2\pi\left(\sum_{i=1}^{N_1}\epsilon p_{1i}\cdot\nabla v_2(\epsilon p_{1i})+\sum_{i=1}^{N_2}\epsilon p_{2i}\cdot\nabla v_1(\epsilon p_{2i})\right)\\
\int_{\mathbb R^2} e^{u_1+u_2} dx=&4\pi((\beta_1-1)(\beta_2-1)-(N_1+1)(N_2+1))
-2\pi\sum_{i=1}^{N_1}\epsilon p_{1i}\cdot\nabla v_2(\epsilon p_{1i})\\
&-2\pi\sum_{i=1}^{N_2}\epsilon p_{2i}\cdot\nabla v_1(\epsilon p_{2i})
\end{split}
\end{equation*}
\end{lemma}
\begin{proof}
In terms of $(v_1,v_2)$, \eqref{CS-5} becomes
\begin{equation*}
\begin{cases}
\Delta v_1+e^{u_2}(1-e^{u_1})=0,\\
\Delta v_2+e^{u_1}(1-e^{u_2})=0.
\end{cases}
\end{equation*}Multiplying the first equation by $x\cdot \nabla (v_2+f_{2\epsilon})$, the second equation by $x\cdot \nabla (v_1+f_{1\epsilon})$, integrating over $B_R$ and summing them up, we get
\begin{eqnarray*}
&&\int_{\partial B_R}(x\cdot \nu)(\nabla v_1\cdot \nabla v_2)dS-\int_{\partial B_R}(x\cdot\nabla v_1)(\nu\cdot \nabla v_2)dS-\int_{\partial B_R}(x\cdot\nabla v_2)(\nabla v_1\cdot\nu)dS\\
=&&\int_{\partial B_R}(x\cdot \nu)(e^{u_1}+e^{u_2}-e^{u_1+u_2})dS-2\int_{B_R}(e^{u_1}+e^{u_2}-e^{u_1+u_2})dx\\
&&+\int_{B_R}(x\cdot \nabla f_{1\epsilon})\Delta v_2 dx+\int_{B_R}(x\cdot \nabla f_{2\epsilon})\Delta v_1 dx
\end{eqnarray*}with $\nu=\frac{x}{|x|}$. Since $v_i(x)=-2(\beta_i+N_i)\ln|x|+O(1)$ near $\infty$, we can get
\begin{equation*}
LHS=-8\pi(\beta_1+N_1)(\beta_2+N_2)+o(1).
\end{equation*}
We need to estimate the last two terms on the righthand side,
\begin{eqnarray*}
\int_{B_R}(x\cdot \nabla f_{1\epsilon})\Delta v_2 dx=&&\int_{B_R} \sum_{i=1}^{N_1}\frac{2x\cdot(x-\epsilon p_{1i})}{|x-\epsilon p_{1i}|^2}\Delta v_2 dx\\
=&&-8\pi N_1(\beta_2+N_2)+\int_{B_R}\sum_{i=1}^{N_1}2\epsilon p_{1i}\cdot\nabla\ln|x-\epsilon p_{1i}|\Delta v_2 dx\\
=&& -8\pi N_1(\beta_2+N_2)-4\pi\sum_{i=1}^{N_1}\epsilon p_{1i}\cdot \nabla v_2(\epsilon p_{1i})+o(1).
\end{eqnarray*}In getting the last equality, we have used the fact that $\ln |x-\epsilon p_{1i}|$ is the Green function and the decay property of $|\nabla^k v_2|,k=1,2$ at $\infty$. Also one can note that
\begin{equation*}
\int_{\mathbb R^2} e^{u_2}(1-e^{u_1})dx=4\pi(\beta_1+N_1),\quad \int_{\mathbb R^2} e^{u_1}(1-e^{u_2})dx=4\pi(\beta_2+N_2).
\end{equation*}Combining all the estimates above, we can get the desired identities.
\end{proof}
Our main job in this section is to prove the following theorem.
\begin{theorem}\label{mainthm2}Let $p_{11},\cdots,p_{1N_1},p_{21},\cdots,p_{2N_2}$ be given. For any $(\beta_1,\beta_2)$ satisfying \eqref{CS-3} and
\begin{equation}
\label{CS-4}\frac{N_1}{\beta_1+N_1}+\frac{N_2}{\beta_2+N_2}\notin\{\frac{k-1}{k}|k=2,\cdots,\max(N_1,N_2)\}.
\end{equation}
for any compact set $K$, there exists a constant $C=C(\beta_1,\beta_2,N_1,N_2,\max(|p_{ij}|),K)$ independent of $\epsilon$ such that for any solution $(u_1,u_2)$ of \eqref{CS-5}, \eqref{CS-2},
\begin{equation}\label{estimate1}
|u_1-f_{1\epsilon}|_{L^\infty(K)}+|u_2-f_{2\epsilon}|_{L^\infty(K)}\leq C.
\end{equation}
\end{theorem}
We will prove Theorem \ref{mainthm2} by contradiction. If Theorem \ref{mainthm2} is not true. Then we have for some compact set $K$ and $\epsilon_n\rightarrow \epsilon^*\in[0,1]$ such that
$$|u_{1n}-f_{1n}|_{L^\infty(K)}+|u_{2n}-f_{2n}|_{L^\infty(K)}\rightarrow \infty, \text{ as }n\rightarrow \infty,$$ where $f_{in}=\displaystyle\sum_{j=1}^{N_i}2\ln|x-\epsilon_n p_{ij}|,i=1,2$. Recall that for $v_{in}(x)=u_{in}(x)-f_{in}(x),i=1,2$,
\begin{equation*}
\begin{cases}
\Delta v_{1n}+e^{u_{2n}}(1-e^{u_{1n}})=0,\\
\Delta v_{2n}+e^{u_{1n}}(1-e^{u_{2n}})=0,
\end{cases}\text{ in }\mathbb R^2.
\end{equation*}From $u_{1n},u_{2n}<0$ in $\mathbb R^2$, we know that for some fixed large $R>\max |p_{ij}|, K\subset\subset B_R$,
$$\Delta (v_{1n}+\frac 12 |x|^2)\geq 0,\Delta (v_{2n}+\frac 12 |x|^2)\geq 0,\text{ in }B_R.$$ This implies that
$$\max_{B_R} (v_{in}+\frac 12 |x|^2)\leq \max_{\partial B_R} (v_{in}+\frac 12 |x|^2)\leq \max_{\partial B_R}-f_{in}+\frac 12R^2\leq C.$$ By this, one can get $v_{1n},v_{2n}$ are uniformly upper bounded in $B_R$. This allows us to assume that
$$\min_{|x|\leq R}v_{1n}\rightarrow -\infty.$$ By Harnack inequality, we will have $v_{1n}\rightarrow  -\infty$ locally uniformly in $\mathbb R^2$, also $u_{1n}\rightarrow  -\infty$ locally uniformly in $\mathbb R^2$ as $n\rightarrow \infty$.
\begin{lemma}
\label{lembl1} If $v_{1n}$ blows up in a compact set $K$, then $v_{2n}$ blows up in the compact set $K$ too.
\end{lemma}
\begin{proof}
If not, we may assume $|v_{2n}|_{L^\infty(K)}\leq C$. By Harnack inequality and standard elliptic estimates, one can get that $v_{2n}(x)\rightarrow v_2(x)$ uniformly in $C^2_{loc}(\mathbb R^2)$ with
$$\Delta v_2=0,\text{ in }\mathbb R^2$$ as we notice that $u_{1n}\rightarrow -\infty$ locally uniformly. This implies that $u_{2n}\rightarrow u_2$ uniformly in $C^2_{loc}(\mathbb R^2\backslash \{\epsilon^* p_{2j},j=1,\cdots,N_2\})$ with
$$\Delta u_2=4\pi\sum_{j=1}^{N_2}\delta_{\epsilon^* p_{2j}},\text{ in }\mathbb R^2.$$ By Lemma \ref{lemPI}, we need to estimate $\nabla v_{1n}(\epsilon_n p_{1j}),\nabla v_{2n}(\epsilon_n p_{2j})$ to get the uniform bound of $\displaystyle \int_{\mathbb R^2} e^{u_{2n}}dx$. By Green's representation formula for $v_{in},i=1,2$, we have
$$v_{1n}(x)=\frac 1{2\pi}\int_{\mathbb R^2}\ln \frac{|y|}{|x-y|}e^{u_{2n}(y)}(1-e^{u_{1n}(y)})dy+c_{1n},\text{ for some constant } c_{1n}.$$ From this, one gets \begin{eqnarray*}
|\nabla v_{1n}(x)|&\leq& \frac 1{2\pi}\int_{\mathbb R^2} \frac{1}{|x-y|}e^{u_{2n}(y)}(1-e^{u_{1n}(y)})dy\\
&\leq & \frac 1{2\pi}\int_{|y-x|\leq 1} \frac{1}{|x-y|}dy+ \frac 1{2\pi}\int_{|x-y|\geq 1} e^{u_{2n}(y)}(1-e^{u_{1n}(y)})dy\\
&\leq & C_1+C_2\leq C<\infty, \text{ independent of } n.
\end{eqnarray*} The argument for $|\nabla v_{2n}|$ is just the same. This indicates that $\displaystyle \int_{\mathbb R^2} e^{u_{2n}}dx$ is uniformly bounded. As $u_2$ is harmonic in $\mathbb R^2\backslash B_{R_0}$, for $R_0\geq \max |p_{2j}|+1$, we get $$\bar u_2(R)=\frac 1{2\pi}\int_{\partial B_R}u_2 dS=\bar u_2(R_0),R\geq R_0.$$ Now we have
\begin{eqnarray*}
\int_{|x|\geq R_0} e^{u_{2}}dx\geq \int_{R_0}^\infty 2\pi r^2e^{\bar u_2(r)}dr=+\infty
\end{eqnarray*}which contradicts to
\begin{eqnarray*}
\int_{\mathbb R^2} e^{u_{2}}dx\leq\lim_{n\rightarrow\infty}\inf\int_{\mathbb R^2} e^{u_{2n}}dx<\infty.
\end{eqnarray*} This means $v_{2n}$ must blow up simultaneously.
\end{proof}

\begin{lemma}
\label{lembl2} Let $x_{in}$ be the maximum point of $u_{in},i=1,2$ respectively. Then $|x_{in}|\rightarrow \infty$ as $n\rightarrow\infty$ and $u_{in}(x_{in})\geq -C$ uniformly for some constant $C>0$.
\begin{proof}
By the previous proof in Lemma \ref{lembl1}, we know that $u_{1n},u_{2n}\rightarrow -\infty$ locally uniformly in $\mathbb R^2$. Suppose $x_{1n}$ is uniformly bounded. Then we will have $u_{1n}(x) \rightarrow -\infty$ uniformly in $\mathbb R^2$. Also by Green's representation formula, we have for $x\in K\subset\subset \mathbb R^2$,
\begin{eqnarray*}
|\nabla v_{1n}(x)|&\leq& \frac 1{2\pi}\int_{\mathbb R^2} \frac{1}{|x-y|}e^{u_{2n}(y)}(1-e^{u_{1n}(y)})dy\\
&\leq & \frac 1{2\pi}\int_{|y-x|\leq R} \frac{1}{|x-y|}e^{u_{2n}(y)}(1-e^{u_{1n}(y)})dy+ \frac 1{2\pi R}\int_{|x-y|\geq R} e^{u_{2n}(y)}(1-e^{u_{1n}(y)})dy\\
&\leq & o(1)R+O(R^{-1}).
\end{eqnarray*}Let $n\rightarrow \infty$, then $R\rightarrow\infty$, one has $\nabla v_{1n}\rightarrow 0$ locally  uniformly in $\mathbb R^2$. The same estimate also holds for $\nabla v_{2n}$. Then by Lemma \ref{lemPI}, we have
\begin{equation*}
4\pi[(\beta_1-1)(\beta_2-1)-(N_1+1)(N_2+1)]+o(1)=\int_{\mathbb R^2}e^{u_{1n}+u_{2n}}dx\leq e^{u_{1n}(x_{1n})}\int_{\mathbb R^2}e^{u_{2n}}dx\rightarrow 0
\end{equation*}which contradicts to $(\beta_1-1)(\beta_2-1)-(N_1+1)(N_2+1)\geq c_0>0$. This implies $x_{in}\rightarrow \infty$. We also note that by Lemma \ref{lemPI},
$$\int_{\mathbb R^2}e^{u_{1n}+u_{2n}}dx\leq 4\pi e^{u_{1n}(x_{1n})}(\beta_1\beta_2-N_1N_2-\beta_2-N_2)+o(1)\leq C_0>0$$ which means $\displaystyle\max_{\mathbb R^2}u_{1n}\geq -C$ uniformly. The same conclusion holds also for $u_{2n}$.
\end{proof}
\end{lemma}
Now by Lemma \ref{lembl2}, we set $r_n=|x_{1n}|^{-1}$, then $r_n\rightarrow 0$ as $n\rightarrow \infty$. Set
\begin{equation}\label{BL-4}
w_{in}(x)=u_{in}\left(\frac{x}{r_n}\right)-2\ln r_n.
\end{equation}
Then $w_{in}$ satisfies
\begin{equation}\label{BL-1}\begin{cases}
\Delta w_{1n}+e^{w_{2n}}(1-r_n^2e^{w_{1n}})=\displaystyle 4\pi\sum_{i=1}^{N_1}\delta_{r_n\epsilon_n p_{1i}},\\
\Delta w_{2n}+e^{w_{1n}}(1-r_n^2e^{w_{2n}})=\displaystyle 4\pi\sum_{i=1}^{N_2}\delta_{r_n\epsilon_n p_{2i}},
\end{cases}\text{ in } \mathbb R^2.
\end{equation}Obviously, there holds $r_n\epsilon_n p_{ij}\rightarrow 0,i=1,2$ as $n\rightarrow \infty$. By Theorem \ref{thmBM1} and noting that $w_{1n}(r_nx_{1n})\rightarrow +\infty$ which means the blow up case in Theorem \ref{thmBM1} happens. Since $\displaystyle \lim_{n\rightarrow \infty} r_{n}x_{1n}\rightarrow q\in \mathbb S^1$, there exists a non-empty finite set $S$ of nonzero points such that $w_{in}\rightarrow -\infty,i=1,2$ uniformly on each $K\subset\subset\mathbb R^2\backslash(S\cup\{0\})$ and
$$e^{w_{2n}}(1-r_n^2e^{w_{1n}})\rightarrow \sum_{q\in S}2\pi M_q\delta_q,\quad e^{w_{1n}}(1-r_n^2e^{w_{2n}})\rightarrow \sum_{q\in S}2\pi N_q\delta_q$$ on any $D\subset\subset\mathbb R^2\backslash\{0\}$ with $S\subset D$ in the distribution sense. For any $q\in S$, set $d$ small such that $B_d(q)\cap (S\cup\{0\})=\{q\}$ and
\begin{equation*}
\begin{split}
&M_{q,n}=\frac 1{2\pi}\int_{B_d(q)}e^{w_{2n}}(1-r_n^2e^{w_{1n}})dx,\quad M_{q,n}\rightarrow M_q,\quad \text{as }n\rightarrow\infty,\\
&N_{q,n}=\frac 1{2\pi}\int_{B_d(q)}e^{w_{1n}}(1-r_n^2e^{w_{2n}})dx,\quad N_{q,n}\rightarrow N_q,\quad \text{as }n\rightarrow\infty.
\end{split}
\end{equation*}
Repeating the proof of Lemma \ref{lemPI} over $B_d(q)$ and by  Pohozaev's identity, we have
\begin{equation}
\label{BL-2}\begin{split}
&\int_{B_d(q)}e^{w_{1n}}dx=\pi M_q(N_q-2)+o(1),\quad \int_{B_d(q)}e^{w_{2n}}dx=\pi N_q(M_q-2)+o(1),\\
&\int_{B_d(q)}r_n^2e^{w_{1n}+w_{2n}}dx=\pi (N_qM_q-2M_q-2N_q)+o(1).
\end{split}
\end{equation}
\eqref{BL-2} implies that $M_q,N_q>2$ and $N_qM_q-2M_q-2N_q\geq 0$. Note that
$$\left(\frac{M_q+N_q}{2}\right)^2\geq M_qN_q\geq 2(M_q+N_q).$$ We have $M_q+N_q\geq 8$ and $|S|\leq \frac{\beta_1+\beta_2+N_1+N_2}{4}$.
In the following, we will prove that all $M_q$'s and $N_q$'s are the same respectively. For this purpose, we need to show the local estimates for $w_{1n},w_{2n}$ and the simple blow up property of $w_{1n}, w_{2n}$. For any $q\in S$, for $d$ small enough
$$w_{in}(q_{in})=\max_{B_{d}(q)}w_{in}(x)\rightarrow +\infty,i=1,2.$$ We have the following estimates.
\begin{lemma}
\label{lemsimple1}$$\max_{|x-q|\leq d}(w_{in}(x)+2\ln|x-q_{in}|)\leq C,i=1,2.$$
\end{lemma}
\begin{proof}
If the lemma is not right, without loss of generality, one may assume that
$$w_{1n}(y_n)+2\ln|y_n-q_{1n}|=\max_{|x-q|\leq d}(w_{1n}(x)+2\ln|x-q_{1n}|)\rightarrow +\infty.$$
It is easy to see that $y_n\neq q_{1n}$ and $q_{1n},y_n\rightarrow q$. Set $d_n=|y_n-q_{1n}|$ and
\begin{equation*}
\bar w_{in}(x)=w_{in}(d_n x+q_{1n})+2\ln d_n,i=1,2,\forall |x|\leq \frac d{2d_n}.
\end{equation*}And $\bar w_{in}$ satisfies that
\begin{equation}\label{BL-3}
\begin{cases}
\Delta \bar w_{1n}+e^{\bar w_{2n}}\left(1-\displaystyle r_n^2d_n^{-2}e^{\bar w_{1n}}\right)=0,\\
\Delta \bar w_{2n}+e^{\bar w_{1n}}\left(1-\displaystyle r_n^2d_n^{-2}e^{\bar w_{2n}}\right)=0,
\end{cases}\text{ in }B_{\frac d{2d_n}}(0)
\end{equation}with \begin{eqnarray*}&&\bar w_{1n}(0)=w_{1n}(q_{1n})+2\ln d_n\geq w_{1n}(y_n)+2\ln d_n\rightarrow +\infty,\\ &&\bar w_{1n}\left(\frac{y_n-q_{1n}}{|y_n-q_{1n}|}\right)=w_{1n}(y_n)+2\ln d_n\rightarrow +\infty. \end{eqnarray*}
\par Hence, $\bar w_{1n}$ blows up at 0 and some point $e=\lim_{n\rightarrow +\infty}\frac{y_n-q_{1n}}{|y_n-q_{1n}|}$. By Theorem \ref{thmBM1},  $\bar w_{1n},\bar w_{2n}$ must blow up simultaneously at these two different points. Denote the blow up set of $\bar w_{1n},\bar w_{2n}$ by $S^*=\{z_1,\cdots,z_l\}$, $l\geq 2$. We have
$$e^{\bar w_{2n}}\left(1-\displaystyle r_n^2d_n^{-2}e^{\bar w_{1n}}\right)\rightarrow \sum_{i=1}^l 2\pi m_i\delta_{z_i},\quad e^{\bar w_{1n}}\left(1-\displaystyle r_n^2d_n^{-2}e^{\bar w_{2n}}\right)\rightarrow \sum_{i=1}^l 2\pi n_i\delta_{z_i}.$$
Consider a unit vector $\xi\in\mathbb R^2$. Multiplying the first equation of \eqref{BL-3} by $\xi\cdot \nabla\bar w_{2n}$ and the second equation by $\xi\cdot \nabla\bar w_{1n}$, integrating by parts in $B_d(z_k)$, we get the following Pohozaev's identity,
\begin{equation}
\label{PI-2}
\begin{split}
&\int_{\partial{B_d(z_k)}}(\xi\cdot \nabla \bar w_{2n})(\nu\cdot \nabla\bar w_{1n})dS+\int_{\partial{B_d(z_k)}}(\nu\cdot \nabla \bar w_{2n})(\xi\cdot \nabla\bar w_{1n})dS\\
=&\int_{\partial{B_d(z_k)}}(\xi\cdot\nu)(\nabla \bar w_{2n}\cdot \nabla\bar w_{1n})dS+\int_{\partial{B_d(z_k)}}(\xi\cdot\nu)\left(\frac{r^2_n}{d^2_n}e^{\bar w_{1n}+\bar w_{2n}}-e^{\bar w_{1n}}-e^{\bar w_{2n}}\right)dS
\end{split}
\end{equation}Here $\nu=\frac{x-z_k}{|x-z_k|}$, $B_d(z_k)\cap S^*=z_k$. Using  Green's representation formula of $u_{1n},u_{2n}$, we have for any $x\in K\subset\subset \mathbb R^2\backslash S^*$, for some fixed $p_0\in  \mathbb R^2\backslash S^*$,
\begin{equation*}
\begin{split}
&\bar w_{1n}(x)-\bar w_{1n}(p_0)\\
=&\frac 1{2\pi}\int_{\mathbb R^2}\ln \left|\frac{p_0-y}{x-y}\right|e^{\bar w_{2n}}\left(1-\frac{r_n^2}{d_n^2}e^{\bar w_{1n}(y)}\right)dy+2\sum_{i=1}^{N_1}\ln\left|\frac{d_n x+q_{1n}-r_n\epsilon_n p_{1i}}{d_n p_0 +q_{1n}-r_n\epsilon_n p_{1i}}\right|\\
=& I_1+I_2.
\end{split}
\end{equation*}Since $q_{1n}\rightarrow q\neq 0, r_n\epsilon_n p_{1i}\rightarrow 0$ as $n\rightarrow \infty$, it is easy to find that $I_2\rightarrow 0$ uniformly on $K$. Split the integral domain of $I_1$ into $\cup_{z_k\in S^*} B_r(z_k)$, $B_R(0)\backslash \cup_{z_k\in S^*} B_r(z_k)$ and $\mathbb R^2\backslash B_R(0)$ and denote the corresponding integrals by $J_1,J_2,J_3$ respectively. Obviously,
$$J_1\rightarrow \sum_{i=1}^l m_i\ln \left|\frac{p_0-z_i}{x-z_i}\right|\quad \text{ uniformly in } K.$$
As for $J_2$, noting that $\bar w_{2n}\rightarrow-\infty$ uniformly in $B_R(0)\backslash \cup_{z_k\in S^*} B_r(z_k)$, we have $J_2\rightarrow 0$. For $J_3$, we have  $y\in \mathbb R^2\backslash B_R(0)$, $x\in K$,
$$1-c_1R^{-1}\leq \left|\frac{p_0-y}{x-y}\right|\leq 1+c_2R^{-1},\quad \int_{\mathbb R^2}e^{\bar w_{2n}}\left(1-\frac{r_n^2}{d_n^2}e^{\bar w_{1n}(y)}\right)dy\leq C$$
for some constants $c_1,c_2,C$ independent of $R,n$. This implies $J_3\rightarrow O(R^{-1})$. Let $n\rightarrow \infty$ then $R\rightarrow \infty$. This yields immediately that
$$\bar w_{1n}(x)-\bar w_{1n}(p_0)\rightarrow \sum_{i=1}^l m_i\ln \left|\frac{p_0-z_i}{x-z_i}\right|\quad \text{ uniformly in } K.$$ The same arguments also lead to
$$\bar w_{2n}(x)-\bar w_{2n}(p_0)\rightarrow \sum_{i=1}^l n_i\ln \left|\frac{p_0-z_i}{x-z_i}\right|\quad \text{ uniformly in } K.$$
Set
$$H_k(x)=\sum_{j\neq k}m_j\ln\left|\frac{p_0-z_j}{x-z_j}\right|,\quad H_k^*(x)=\sum_{j\neq k}n_j\ln\left|\frac{p_0-z_j}{x-z_j}\right|.$$ Using Pohozaev's identity \eqref{PI-2} and the fact $\int_{\partial B_d(z_k)}\xi\cdot\nu dS=0$, we can get
\begin{equation*}
\begin{split}
&-\frac{1}{d}\int_{\partial B_d(z_k)}\xi\cdot (m_k\nabla H^*_k+n_k\nabla H_k) dS+\int_{\partial B_d(z_k)}(\xi\cdot\nabla H_k^*)(\nu\cdot\nabla H_k)+(\xi\cdot\nabla H_k)(\nu\cdot\nabla H_k^*)dx\\
=&\int_{\partial B_d(z_k)}(\xi\cdot\nu)(\nabla H_k^*\cdot\nabla H_k)dx+\int_{\partial B_d(z_k)}(\xi\cdot\nu)\left(\frac{r_n^2}{d_n^2}e^{\bar w_{1n}+\bar w_{2n}}-e^{\bar w_{1n}}-e^{\bar w_{2n}}\right)dS+o(1).
\end{split}
\end{equation*}Let $n\rightarrow \infty$ and $d\rightarrow 0$. By noting the arbitrariness of $\xi$, then we get
\begin{equation}\label{simple1}
m_k\nabla H_k^*(z_k)+n_k\nabla H_k(z_k)=0,k=1,\cdots,l.
\end{equation} As $l\geq 2$ and $0,e\in S^*$, without loss of generality, we may assume $z_{1,1}=\max_{k\leq l}z_{k,1}$ and $z_{1,1}>z_{2,1}$ where $z_{k,1}$ stands for the first coordinate of $z_k$. By this choice, we can get $$\sum_{j\neq 1}\frac{(m_1n_j+n_1m_j)(z_{1,1}-z_{j,1})}{|z_1-z_j|^2}>0$$ which contradicts to \eqref{simple1} with $k=1$. This completes the proof of the present lemma.
\end{proof}
Moreover, we also have the following lemma.
\begin{lemma}
\label{lemsimple2}$$\max_{|x-q|\leq d} (w_{1n}(x)+2\ln|x-q_{2n}|)\leq C,\quad \max_{|x-q|\leq d} (w_{2n}(x)+2\ln|x-q_{1n}|)\leq C.$$
\end{lemma}
\begin{proof}
The same as the proof in Lemma \ref{lemsimple1}, we obtain the conclusion by contradiction. Suppose not, we have
\begin{equation*}
w_{1n}(y_n)+2\ln|y_n-q_{2n}|=\max_{|x-q|\leq d} (w_{1n}(x)+2\ln|x-q_{2n}|)\rightarrow +\infty.
\end{equation*}It is obvious that $y_n\rightarrow q$ as $n\rightarrow \infty$. Set
$$\bar w_{in}(x)=w_{in}(d_n x+q_{2n})+2\ln d_n,i=1,2,\forall |x|\leq \frac d{2d_n}$$ where $d_n=|y_n-q_{2n}|$. $\bar w_{in}(x)$ satisfies the same equation as \eqref{BL-3}. Note that
$$\bar w_{1n}\left(\frac{y_n-q_{2n}}{|y_n-q_{2n}|}\right)=w_{1n}(y_n)+2\ln d_n\rightarrow +\infty. $$ Theorem \ref{thmBM1} tells us that the blow-up case happens. Now we need to prove the blow up set contains at least two different points. Denote
$$\lim_{n\rightarrow \infty}\frac{y_n-q_{2n}}{|y_n-q_{2n}|}=q^*\in \mathbb S^1, |q^*|=1.$$ By Theorem \ref{thmBM1}, there exists a sequence of points $z_n\rightarrow q^*$, such that $\bar w_{2n}(z_n)\rightarrow \infty$ and $|z_n|\geq \frac 12$. Since
$$\bar w_{2n}(0)=w_{2n}(q_{2n})+2\ln d_n\geq w_{2n}(d_n z_n+q_{2n})+2\ln d_n=\bar w_{2n}(z_n)\rightarrow +\infty,$$ one gets there are at least two different blow-up points. Then following the same arguments as in Lemma \ref{lemsimple1}, we can prove Lemma \ref{lemsimple2}.
\end{proof}
With the above two lemmas, following the arguments in \cite{BartolucciChenLinTarantello2004}, we now prove the simple blow-up estimates.
\begin{lemma}
\label{lemsimple3}
For any $x\in B_{d}(q),q\in S$, we have
\begin{eqnarray*}
&&\left|w_{1n}(x)-w_{1n}(q_{1n})+\frac{M_{q,n}}{2}\ln(1+e^{w_{1n}(q_{1n})}|x-q_{1n}|^2)\right|\leq C,\\
&&\left|w_{2n}(x)-w_{2n}(q_{2n})+\frac{N_{q,n}}{2}\ln(1+e^{w_{2n}(q_{2n})}|x-q_{2n}|^2)\right|\leq C.
\end{eqnarray*}
\end{lemma}
\begin{proof}
Set $$\bar w_{in}(x)=w_{in}(s_n x+q_{1n})+2\ln s_n,\quad s_n=\exp\left(-\frac{w_{1n}(q_{1n})}{2}\right),\quad x\in B_{2d/s_n}(0),$$
where $B_{4d}(q)\cap S=\{q\}$. Then $w_{in},i=1,2$ satisfies
the same system as \eqref{BL-3} with $d_n$ replaced by $s_n$.
Since $\bar w_{1n}(0)=0$ and $u_{1n}(x)<0$, we have $\frac {r_n}{s_n}\leq 1$. Also, we can apply Harnack inequality and (i) of Theorem \ref{thmBM1} in any compact domain $K$ contains $0$ to get that $\bar w_{1n}\in L^{\infty}(K)$ and either $\bar w_{2n}\in L^{\infty}(K)$ or $\bar w_{2n}\rightarrow -\infty$ locally uniformly.  By the same arguments as in Lemma \ref{lembl1}, we can exclude the case  $\bar w_{2n}\rightarrow -\infty$ if we notice that $\int e^{\bar w_{1n}}dx\leq C<\infty$. Now we can take a subsequence such that $\bar w_{kn}\rightarrow \bar w_{k}$ in $C^2_{loc}(\mathbb R^2)$ with $\bar w_{k}$ satisfying
\begin{equation*}
\begin{cases}
\Delta \bar w_1+e^{\bar w_2}(1-c^2e^{\bar w_1})=0,\\
\Delta \bar w_2+e^{\bar w_1}(1-c^2e^{\bar w_2})=0,
\end{cases}\text{ in }\mathbb R^2
\end{equation*}for some constant $0\leq c\leq 1.$ By the $L^1$ integrability of $e^{\bar w_1},e^{\bar w_2},c^2e^{\bar w_1+\bar w_2}$ in $\mathbb R^2$ and Green's representation formula for $\bar w_1,\bar w_2$, we must have
$$\frac 1{2\pi}\int_{\mathbb R^2}e^{\bar w_2}(1-c^2e^{\bar w_1})dx>2,\quad \frac 1{2\pi}\int_{\mathbb R^2}e^{\bar w_1}(1-c^2e^{\bar w_2})dx>2.$$ This implies we can choose $R_0$ large enough and $\delta_0>0$ small enough such that
\begin{equation*}
\frac{1}{2\pi}\int_{B_{R_0}}e^{\bar w_{2n}}(1-r_n^2/s_n^2 e^{\bar w_{1n}})dx>2+4\delta_0,\quad \frac{1}{2\pi}\int_{B_{R_0}}e^{\bar w_{1n}}(1-r_n^2/s_n^2 e^{\bar w_{2n}})dx>2+4\delta_0.
\end{equation*}
We claim that
\begin{equation}\label{claim1}
\bar w_{1n}(x)\leq -(2+\delta_0)\ln |x|+C_0,\text{ for }2R_0\leq |x|\leq \frac{d}{s_n}.
\end{equation}
By Green's representation formula and the same arguments as in Lemma \ref{lemsimple1}, one can get for $2R_0\leq |x|\leq d/s_n$
\begin{equation*}
\bar w_{1n}(x)=\frac{1}{2\pi}\int_{|y|\leq d/s_n}\ln \frac{|y|}{|x-y|}e^{\bar w_{2n}(y)}(1-r_n^2/s_n^2 e^{\bar w_{1n}(y)}) dy+O(1)
\end{equation*} where $O(1)$ always denotes some uniform bounded term from now on.
Split the integral into three domains: $D_1=\{|y|\leq R_0\}$, $D_2=\{|y-x|\leq \frac{|x|}2,R_0\leq |y|\leq d/s_n\}$, $D_3=\{|y-x|\geq \frac{|x|}2,R_0\leq y\leq d/s_n\}$. In $D_1$, we have $\frac{|x|}2\leq |x-y|\leq \frac{3|x|}{2}$, then
\begin{eqnarray*}
\int_{D_1}\ln \frac{|y|}{|x-y|} e^{\bar w_{2n}(y)}(1-r_n^2/s_n^2e^{\bar w_{1n}(y)})dy
=O(1)-\int_{|y|\leq R_0} e^{\bar w_{2n}(y)}(1-r_n^2/s_n^2e^{\bar w_{1n}(y)})dy\ln|x|.
\end{eqnarray*}
In $D_2$, we have $\frac 32 |x|\geq |y|\geq \frac{|x|}2$. From Lemma \ref{lemsimple1} and Lemma \ref{lemsimple2}, one can get
\begin{eqnarray*}
\bar w_{kn}(x)+2\ln |x|=w_{kn}(z)+2\ln |z-q_{1n}|\leq C, \quad z=s_n x+q_{1n}.
\end{eqnarray*}Substituting this into the integration on $D_2$, we have
\begin{eqnarray*}
\int_{D_2}\ln \frac{|y|}{|x-y|} e^{\bar w_{2n}(y)}(1-r_n^2/s_n^2e^{\bar w_{1n}(y)})dy\leq O(1)+\frac C{|x|^2}\left(\pi (|x|/2)^2\ln |x|-\pi\int_0^{\frac{|x|}{2}}\ln r d r^2\right)=O(1)
\end{eqnarray*}In $D_3$, we have $|x-y|\geq \frac{|y|}3$. This implies
\begin{eqnarray*}
\int_{D_3}\ln \frac{|y|}{|x-y|} e^{\bar w_{2n}(y)}(1-r_n^2/s_n^2e^{\bar w_{1n}(y)})dy\leq C.
\end{eqnarray*}
Combining the above three estimates and the choice of $R_0$, we prove the claim \eqref{claim1}. The same estimates also hold true for $\bar w_{2n}$. By the claim, we have
\begin{equation*}
\int_{|y|\leq d/s_n}\ln |y| e^{\bar w_{2n}}(1-r_n^2/s_n^2 e^{\bar w_{1n}})dy= O(1)
\end{equation*}which implies
\begin{equation*}
\bar w_{1n}(x)=-M_{q,n}\ln |x|+O(1)+\frac 1{2\pi}\int_{|y|\leq d/s_n}\ln \frac{|x|}{|x-y|} e^{\bar w_{2n}}(1-r_n^2/s_n^2 e^{\bar w_{1n}})dy,\quad 2R_0\leq |x|\leq \frac{d}{s_n}.
\end{equation*}We now estimate the last term in the above equation. Split the integral into three parts: $K_1=\{|y|\leq \frac {|x|}2\}$, $K_2=\{|y-x|\leq \frac{|x|}{2},|y|\leq d/s_n\}$, $K_3=\{|y-x|\geq \frac{|x|}2,\frac {|x|}2\leq |y|\leq d/s_n\}$.
In $K_1$, we have $2/3\leq \frac{|x|}{|y-x|}\leq 2$. Thus the integration on $K_1$ is $O(1)$. In $K_2$, we have $|y|\geq |x|/2\geq R_0$, therefore
\begin{eqnarray*}
\left|\int_{K_2}\ln \frac{|x|}{|x-y|} e^{\bar w_{2n}}(1-r_n^2/s_n^2 e^{\bar w_{1n}})dy\right|\leq c|x|^{-2-\delta_0}\left|\int_{|y-x|\leq |x|/2}\ln \frac{|x|}{|x-y|}dy\right|\leq C.
\end{eqnarray*}
In $K_3$, we have $\frac{1}{2|y|}\leq \frac{|x|}{|y-x|}\leq 2$, then
$$\left|\int_{K_3}\ln \frac{|x|}{|x-y|} e^{\bar w_{2n}}(1-r_n^2/s_n^2 e^{\bar w_{1n}})dy\right|\leq C+\left|\int_{K_3}\ln |y| e^{\bar w_{2n}}(1-r_n^2/s_n^2 e^{\bar w_{1n}})dy\right|\leq C_1.$$
The three estimates yield that
$$\bar w_{1n}(x)+M_{q,n}\ln |x|=O(1),\quad \text{for }2R_0\leq |x|\leq d/s_n.$$ Combining with the uniform bound of $\bar w_{1n}$ in $B_{2R_0}$, we finish the proof of simple blow-up estimates.
\end{proof}
Now we can prove all the $M_q$'s and $N_q$'s are the same respectively.
\begin{lemma}
\label{lemsame}
Let $(w_{1n},w_{2n})$ and $M_q,N_q$ be defined as before. Then $M_p=M_q,N_p=N_q$ for all $p,q\in S$.
\end{lemma}
\begin{proof}
Step 1. $\exists q\in S$, there holds $M_qN_q>2(M_q+N_q)$. Recall $x_{1n}$ is the maximum point of $u_{1n}$ and $q_0=\lim_{n\rightarrow \infty} \frac{x_{1n}}{|x_{1n}|}\in S$, $r_n=1/|x_{1n}|$, $|q_0|=1$. Set
$$\bar u_{1n}(x)=u(x+x_{1n}),\bar u_{2n}(x)=u(x+x_{1n}).$$ By Lemma \ref{lembl2}, one gets $\bar u_{1n}(0)\geq -C$. Also as $|x_{1n}|\rightarrow \infty$, by the standard elliptic estimates, we have $\bar u_{1n}\rightarrow \bar u_1$ in $C^2_{loc}(\mathbb R^2)$. By Theorem \ref{thmBM1}, Harnack inequality and the arguments in Lemma \ref{lembl1}, we shall have $\bar u_{2n}\rightarrow \bar u_2$ in $C^2_{loc}(\mathbb R^2)$ too. This yields that
\begin{equation*}
\begin{cases}
\Delta \bar u_1+e^{\bar u_2}(1-e^{\bar u_1})=0,\\
\Delta \bar u_2+e^{\bar u_1}(1-e^{\bar u_2})=0,
\end{cases}\text{ in }\mathbb R^2
\end{equation*} with $\int_{\mathbb R^2}e^{\bar u_1}+e^{\bar u_2}dx< \infty$. This means $\bar u_1,\bar u_2$ are non-topological solutions. Denote
$$\int_{\mathbb R^2}e^{\bar u_2}(1-e^{\bar u_1})=2\pi M_0,\int_{\mathbb R^2}e^{\bar u_1}(1-e^{\bar u_2})=2\pi N_0.$$ By Lemma \ref{lemPI}, we have
$$\int_{\mathbb R^2}e^{\bar u_1+\bar u_2}dx=\pi(M_0N_0-2M_0-2N_0)>0. $$ By the Fatou's lemma, we have $M_{q_0}\geq M_0,N_{q_0}\geq N_0$ and
$$\frac 1{M_{q_0}}+\frac 1{N_{q_0}}\leq \frac{1}{M_0}+\frac 1{N_0}<\frac 12.$$ This ends the proof of the first step.
\par Step 2. If $M_qN_q>2(M_q+N_q)$, we have $w_{in}(q_{in})=-\ln r_n^2+O(1).$ In fact, by \eqref{BL-2}, we have
$$\int_{B_d(q)}r_n^2 e^{w_{1n}+w_{2n}}dx=\pi(M_qN_q-2M_q-2N_q)+o(1)\geq c_0>0.$$ Thus
$$\max_{B_d(q)}(r_n^2 e^{w_{1n}})\int_{B_d(q)}e^{w_{2n}}dx\geq \int_{B_d(q)}r_n^2 e^{w_{1n}+w_{2n}}dx\geq c_0. $$
 This means that $w_{1n}(q_{1n})+2\ln r_n\geq -C$ uniformly. Also noting $w_{1n}+2\ln r_n\leq 0$, we have $w_{kn}(q_{kn})=-\ln r_n^2+O(1)$.
\par Step 3. For any $p,q\in S$,
$$|w_{kn}(x)-w_{kn}(y)|\leq C,\quad \forall x\in \partial B_d(q),y\in \partial B_d(p),k=1,2.$$ By Green's representation formula, we have
\begin{eqnarray*}
|w_{1n}(x)-w_{1n}(y)|\leq \frac 1{2\pi}\left|\int_{\mathbb R^2}\ln \frac{|x- z|}{|y- z|}e^{w_{2n}(z)}(1-r_n^2 e^{w_{1n}(z)})dz\right|+2\sum_{i=1}^{N_1}\left|\ln\frac{|x-r_n\epsilon_n p_{1i}|}{|y-r_n\epsilon_n p_{1i}|}\right|.
\end{eqnarray*}The uniform bound of the second term on the righthand side is obvious. We split the first term into two integral domain to estimate it. In $|z-x|\leq \frac d2$ or $|z-y|\leq \frac d2$, we have $e^{w_{2n}(z)}(1-r_n^2 e^{w_{1n}(z)})\rightarrow 0$ uniformly, therefore
$$\int_{\{|z-x|\leq \frac d2\}\cup\{|z-y|\leq \frac d2\}}\ln \frac{|x- z|}{|y- z|}e^{w_{2n}(z)}(1-r_n^2 e^{w_{1n}(z)})dz\rightarrow 0.$$
In $\mathbb R^2\backslash (\{|z-x|\leq \frac d2\}\cup\{|z-y|\leq \frac d2)$, we have if $|z|\geq 2\max (|x|,|y|)$, then $\frac 13\leq \frac{|x-z|}{|y-z|}\leq 3$; if $|z|\leq 2\max (|x|,|y|)$, then $\frac {d}{6|y|}\leq\frac{|x-z|}{|y-z|}\leq \frac{6|x|}{d}$. Therefore, we have
$$\int_{\mathbb R^2\backslash(\{|z-x|\leq \frac d2\}\cup\{|z-y|\leq \frac d2\})}\ln \frac{|x- z|}{|y- z|}e^{w_{2n}(z)}(1-r_n^2 e^{w_{1n}(z)})dz=O(1).$$
\par Step 4. $M_p=M_q,N_p=N_q$ holds for all $p,q\in S$. Let $q_0$ be the point as in Step 1, i.e. $M_{q_0}N_{q_0}>2(M_{q_0}+N_{q_0})$. If $M_pN_p>2(M_p+N_p)$, then we have by Step 2, $w_{kn}(q_{0,kn}),w_{kn}(p_{kn})=-2\ln r_n+O(1),k=1,2$. Picking up $ x\in \partial B_d(q),y\in \partial B_d(p)$, we have by Lemma \ref{lemsimple3}
\begin{eqnarray*}
w_{1n}(x)-w_{1n}(y)=(1-\frac{M_{q_0,n}}{2})(-\ln r_n^2)-(1-\frac{M_{p,n}}{2})(-\ln r_n^2)+O(1).
\end{eqnarray*}By Step 3, we must have $M_p=M_q$. Also we can get $N_p=N_q$ by investigating $w_{2n}(x)-w_{2n}(y)$. If $M_pN_p=2(M_p+N_p)$, then without loss of generality, we may assume $M_p<M_{q_0}$. By Lemma \ref{lemsimple3} and Step 1, we also have
\begin{eqnarray*}
w_{1n}(x)-w_{1n}(y)\leq C+(M_{q_0,n}-M_{p,n})\ln r_n\rightarrow -\infty
\end{eqnarray*}contradicts to Step 3. This finishes the proof.
\end{proof}
By Lemma \ref{lemsame}, we now denote $M,N$ the mass at  $q\in S$ instead of $M_q,N_q$ respectively.
In the following two lemmas, we show the concentration may occur only for special values of $\alpha_1,\alpha_2$. The following lemma shows that there is no concentration at the origin.
\begin{lemma}
\label{lemorigin}For each constant $0<s<\min_{q\in S}|q|$, we have
\begin{equation}\label{origin}
\lim_{n\rightarrow\infty}\left[\int_{|x|\leq s}e^{w_{1n}}(1-r_n^2 e^{w_{2n}})dx+\int_{|x|\leq s}e^{w_{2n}}(1-r_n^2 e^{w_{1n}})dx\right]=0
\end{equation}  and $w_{1n}$, $w_{2n}\rightarrow-\infty$ uniformly on $B_s(0)$. Moreover,  $|S|\geq 2$, $MN(|S|-1)=2NN_1+2MN_2$.
\end{lemma}
\begin{proof}
Consider a small number $s< \min_{q\in S} |q|$. Suppose that
$$2\pi M_0=\lim_{n\rightarrow \infty}\int_{|x|\leq s}e^{w_{2n}}(1-r_n^2 e^{w_{1n}})dx>0,2\pi N_0=\lim_{n\rightarrow \infty}\int_{|x|\leq s}e^{w_{1n}}(1-r_n^2 e^{w_{2n}})dx.$$ Set $$\tilde w_{in}(x)=w_{in}(x)-f_{i,r_n\epsilon_n},f_{i,r_n\epsilon_n}(x)=2\sum_{j=1}^{N_i}\ln |x-r_n\epsilon_n p_{ij}|,i=1,2.$$
By previous discussion, we have $$\tilde w_{in}(x)\rightarrow -\infty \text{ uniformly for any }K\subset\subset \mathbb R^2\backslash(S\cup\{0\}).$$ By the Green's representation formula for $\tilde w_{in}$ and the standard elliptic estimates, one can get
\begin{equation*}
\begin{split}
&\tilde w_{1n}(x)-c_{1n}\rightarrow -M_0\ln |x|-\sum_{q\in S}M\ln |x-q|,\\
&\tilde w_{2n}(x)-c_{2n}\rightarrow -N_0\ln |x|-\sum_{q\in S}N\ln |x-q|,
\end{split}
\end{equation*} in $C^2_{loc}(\mathbb R^2\backslash(S\cup\{0\}))$. Recall the Pohozaev-type identity as in Lemma \ref{lemPI},
\begin{equation*}
\begin{split}
&\int_{\partial \Omega}(x\cdot \nu)(\nabla \tilde w_{1n}\cdot\nabla \tilde w_{2n})dS-\int_{\partial \Omega}(x\cdot\nabla \tilde w_{1n})(\nu\cdot\nabla \tilde w_{2n})dS-\int_{\partial \Omega}(x\cdot\nabla \tilde w_{2n} )(\nabla \tilde w_{1n}\cdot\nu)dS\\
=&\int_{\partial \Omega}(x\cdot\nu)(e^{w_{1n}}+e^{w_{2n}}-r_n^2e^{w_{1n}+w_{2n}})dS-2\int_{\Omega}(e^{w_{1n}}+e^{w_{2n}}-r_n^2e^{w_{1n}+w_{2n}})dx\\
-&\int_{\Omega}(x\cdot \nabla f_{1,r_n\epsilon_n})(e^{w_{1n}}-r_n^2e^{w_{1n}+w_{2n}})dx-\int_{\Omega}(x\cdot\nabla f_{2,r_n\epsilon_n})(e^{w_{2n}}-r_n^2e^{w_{1n}+w_{2n}})dx.
\end{split}
\end{equation*}Set
\begin{equation*}
\begin{split}
& I_n=\int_{\partial \Omega}(x\cdot \nu)(\nabla \tilde w_{1n}\cdot\nabla \tilde w_{2n})dS-\int_{\partial \Omega}(x\cdot\nabla \tilde w_{1n})(\nu\cdot\nabla \tilde w_{2n})dS-\int_{\partial \Omega}(x\cdot\nabla \tilde w_{2n} )(\nabla \tilde w_{1n}\cdot\nu)dS\\
& J_{1n}=\int_{\Omega}(x\cdot \nabla f_{1,r_n\epsilon_n})(e^{w_{1n}}-r_n^2e^{w_{1n}+w_{2n}})dx,J_{2n}=\int_{\Omega}(x\cdot\nabla f_{2,r_n\epsilon_n})(e^{w_{2n}}-r_n^2e^{w_{1n}+w_{2n}})dx.
\end{split}
\end{equation*}
First we take  $\Omega=\{|x|\leq s\}$, then
$$I_n=-2\pi M_0N_0+o(1)$$ as $n\rightarrow\infty$. Moreover,
\begin{equation*}
\begin{split}
J_{1n}=&4\pi N_1N_0+\sum_{i=1}^{N_1}\int_{|x|\leq s/r_{n}}\frac{2(x-\epsilon_n p_{1i})\cdot\epsilon_n p_{1i}}{|x-\epsilon_n p_{1i}|^2}e^{u_{1n}}(1-e^{u_{2n}})dx+o(1)\\
=&4\pi N_1N_0+ o(1)\int_{|x|\leq R}\frac{2(x-\epsilon_n p_{1i})\cdot\epsilon_n p_{1i}}{|x-\epsilon_n p_{1i}|^2}dx+R^{-1}\int_{|x|\geq R}e^{u_{1n}}(1-e^{u_{2n}})dx+o(1)\\
=&4\pi N_1N_0+o(1)
\end{split}
\end{equation*}as $n\rightarrow \infty$ and then $R\rightarrow \infty$. Also we have $J_{2n}=4\pi N_2M_0+o(1)$. By the above estimates, we have by Pohozaev's identity,
\begin{equation*}
-2\pi M_0N_0=-4\pi N_2M_0-4\pi N_1N_0-4\pi M_0-4\pi N_0-2\int_{|x|\leq s}r_n^2e^{w_{1n}+w_{2n}}dx+o(1).
\end{equation*}Therefore we have
\begin{equation}
\label{origin1}
M_0N_0\geq 2(N_2M_0+N_1N_0+M_0+N_0).
\end{equation}From \eqref{origin1}, we get $N_0>0$. Next we take $\Omega=\cup_{q\in S}B_r(q), r$ small enough. Note that
\begin{equation*}
\begin{split}
&\nabla \tilde w_{1n}\rightarrow\nabla \tilde w_{1}= -M\frac{x-q}{|x-q|^2}+\nabla H_{1q}(x),\nabla \tilde w_{2n}\rightarrow \nabla \tilde w_2=-N\frac{x-q}{|x-q|^2}+\nabla H_{2q}(x),\\
&\nabla H_{1q}(x)=-M_0\frac{x}{|x|^2}-\sum_{p\in S\backslash\{q\}}M\frac{x-p}{|x-p|^2},\nabla H_{2q}(x)=-N_0\frac{x}{|x|^2}-\sum_{p\in S\backslash\{q\}}N\frac{x-p}{|x-p|^2}.
\end{split}
\end{equation*}As $n\rightarrow \infty$, we have
\begin{equation*}
\begin{split}
I_n\rightarrow &\int_{\partial \Omega}(x\cdot \nu)(\nabla \tilde w_{1}\cdot\nabla \tilde w_{2})dS-\int_{\partial \Omega}(x\cdot\nabla \tilde w_{1})(\nu\cdot\nabla \tilde w_{2})dS-\int_{\partial \Omega}(x\cdot\nabla \tilde w_{2} )(\nabla \tilde w_{1}\cdot\nu)dS\\
=& -2\pi MN|S|+\sum_{q\in S}\int_{|x-q|=r}\frac{N}{r}q\cdot \nabla H_{1q}+\frac Mr q\cdot\nabla H_{2q} dS +O(r)\\
=& -2\pi MN|S|-2\pi NM_0|S|-2\pi M N_0|S|-\sum_{q\in S}\sum_{p\neq q}4\pi MN\frac{q\cdot(q-p)}{|q-p|^2}+O(r)\\
=&-2\pi|S|(MN|S|+NM_0+MN_0)+O(r).
\end{split}
\end{equation*}
From \eqref{BL-2}, we have
\begin{equation*}
\begin{split}
&2\int_{\Omega}(e^{w_{1n}}+e^{w_{2n}}-r_n^2e^{w_{1n}+w_{2n}})dx
\\=&4\pi |S|M+4\pi |S|N+2\pi|S|(MN-2M-2N)+o(1)
\end{split}
\end{equation*}
Moreover, we have
\begin{equation*}
\begin{split}
J_{1n}=&4\pi N_1N|S|+\sum_{i=1}^{N_1}\int_\Omega \frac{2r_n\epsilon_n p_{1i}\cdot(x-r_n\epsilon_n p_{1i})}{|x-r_n\epsilon_n p_{1i}|^2}(e^{w_{1n}}-r_n^2e^{w_{1n}+w_{2n}})dx\\
=&4\pi NN_1|S|+o(1).
\end{split}
\end{equation*}Also we have $J_{2n}=4\pi MN_2|S|+o(1)$. Combining the above estimates, the Pohozaev's identity implies that
\begin{equation}
\label{origin2}
MN(|S|-1)+MN_0+NM_0=2NN_1+2MN_2.
\end{equation}
By \eqref{origin1}, one has
$$M_0N_0>2N_2M_0+2M_0\Rightarrow N_0>2(N_2+1).$$
The same arguments also yield $M_0>2(N_1+1)$. Substituting $M_0,N_0$ by these two inequalities we get
$$MN_0+NM_0>2(N_2+1)M+2(N_1+1)N$$ contradicts to \eqref{origin2}. We must have $M_0=N_0=0$ which implies that
\begin{equation*}
MN(|S|-1)=2NN_1+2MN_2.
\end{equation*} This also implies $|S|\geq 2$.
\par The last thing we shall prove is $w_{in}\rightarrow -\infty$ uniformly on $B_s(0)$ for $i=1,2$. By Green's representation formula for $u_{1n}$, we have, for $x\leq \frac{s}{2r_n}$,
\begin{equation*}
\begin{split}
u_{1n}(x)=f_{1n}(x)+C_{1n}+\frac 1{2\pi}\int_{|y|\leq \frac{s}{r_n}}\ln\frac{|y|}{|x-y|}e^{u_{2n}}(1-e^{u_{1n}})dy +O(1).
\end{split}
\end{equation*}
Standard arguments(\cite{BartolucciChenLinTarantello2004}) shows that
$$u_{1n}(x)-f_{1n}(x)-C_{1n}=o(1)\ln r_n+O(1), \quad \text{for }|x|\leq \frac{s}{2r_n}.$$
Then it follows from Lemma \ref{lemsimple3} that $C_{1n}=(M_{q,n}+2N_1+o(1))\ln r_n+O(1)$, and hence, $w_{1n}=f_{1,r_n\epsilon_n}(x)+(M_{q,n}-2+o(1))\ln r_n+O(1)$ for $|x|\leq \frac s2$. This ends the proof of present lemma.
\end{proof}
\begin{remark}
Even when all $M_q$'s and $N_q$'s are not the same, it is still true that there is no concentration at origin by repeating the computation in Lemma \ref{lemorigin}.
\end{remark}
Next lemma shows that there is no concentration at $\infty$.
\begin{lemma}\label{leminfty}
For any $R>\max_{q\in S}|q|$, $w_{in}\rightarrow -\infty$ uniformly in $|x|\geq R$ and
$$\lim_{n\rightarrow \infty}\int_{|x|\geq R}e^{w_{1n}}(1-r_n^2 e^{w_{2n}})dx+\int_{|x|\geq R}e^{w_{2n}}(1-r_n^2 e^{w_{1n}})dx=0.$$
Moreover, we have $MN(|S|+1)=2\beta_1N+2\beta_2M$.
\end{lemma}
\begin{proof}
Set
$$\varphi_{in}(x)=w_{in}\left(\frac x{|x|^2}\right)-2\beta_i\ln |x|,i=1,2.$$
Then $\varphi_{in}$ satisfies
\begin{equation*}
\begin{cases}
&-\Delta \varphi_{1n}=|x|^{2\beta_2-4}e^{\varphi_{2n}}(1-r_n^2|x|^{2\beta_1}e^{\varphi_{1n}}),\\
&-\Delta \varphi_{2n}=|x|^{2\beta_1-4}e^{\varphi_{1n}}(1-r_n^2|x|^{2\beta_2}e^{\varphi_{2n}}),
\end{cases}\text{ in } |x|<R
\end{equation*}for $R<(\max_{q\in S}|q|)^{-1}$. Denote $\tilde S=\{\frac q{|q|^2}|q\in S\}$. Then by previous arguments, we have $\varphi_{in}\rightarrow-\infty,i=1,2$ uniformly for any compact set $K\subset\mathbb R^2\backslash (\tilde S\cup \{0\})$.

For  $0<s<\frac 14(\max_{q\in S} |q|)^{-1}$, set
\begin{eqnarray*}&&\lim_{n\rightarrow\infty}\inf\int_{|x|\leq s}|x|^{2\beta_2-4}e^{\varphi_{2n}}(1-r_n^2|x|^{2\beta_1}e^{\varphi_{1n}})dx=2\pi L_1,\\
&&\lim_{n\rightarrow\infty}\inf\int_{|x|\leq s}|x|^{2\beta_1-4}e^{\varphi_{1n}}(1-r_n^2|x|^{2\beta_2}e^{\varphi_{2n}})dx=2\pi L_2.
 \end{eqnarray*}
We need to show $\varphi_{in}\rightarrow -\infty$ uniformly in $B_s$. If not, we may assume $\max_{|x|\leq s}\varphi_{2n}\geq -C$. Then we must have $\max_{|x|\leq s} (\varphi_{1n},\varphi_{2n})\rightarrow +\infty$. Otherwise, as $\beta_1>1$, we will have by Harnack inequality that $\varphi_{2n}$ is uniformly bounded in $B_s$ which contradicts to $\varphi_{2n}\rightarrow -\infty$ on $\partial B_s$. Denote $\lambda_{in}=\varphi_{in}(y_{in})=\max_{|x|\leq s}\varphi_{in}$. Set
$$t_n=\min\{e^{-\frac{\lambda_{1n}}{2\beta_1-2}},e^{-\frac{\lambda_{2n}}{2\beta_2-2}}\}=e^{-\frac{\lambda_{1n}}{2\beta_1-2}}\rightarrow 0.$$
Then we have $y_{1n}\rightarrow 0$ and $\lambda_{1n}\rightarrow +\infty$. Next, we shall discuss in two cases:
\begin{itemize}
\item[Case 1.] $\frac{|y_{1n}|}{t_n}\leq C$. Set
$$\bar \varphi_{in}(x)=\varphi_{in}(t_n  x)-\lambda_{in},i=1,2.$$
Then we have $\bar \varphi_{in}\leq 0,i=1,2.$ and
\begin{equation*}
\begin{cases}
\displaystyle -\Delta \bar \varphi_{1n}=e^{\lambda_{2n}+2(\beta_2-1)\ln t_n}|x|^{2\beta_2-4}e^{\bar \varphi_{2n}}(1-r_n^2t_n^2|x|^{2\beta_1}e^{\bar \varphi_{1n}}),\\
\displaystyle-\Delta \bar \varphi_{2n}=|x|^{2\beta_1-4}e^{\bar \varphi_{1n}}(1-r_n^2t_n^{2\beta_2}e^{\lambda_{2n}}|x|^{2\beta_2}e^{\bar \varphi_{2n}}),
\end{cases}\text{ in } |x|\leq s/t_n.
\end{equation*}Also we have $\bar \varphi_{1n}(x)\leq \bar\varphi_{1n}(\frac{y_{1n}}{t_n})=0$. Then by standard elliptic estimates and $\lambda_{2n}+2(\beta_2-1)\ln t_n\leq 0$, we have $\bar \varphi_{1n}\rightarrow \varphi_1$ in $W^{2,\gamma}_{loc}(\mathbb R^2)$ for some $\gamma>1$. As for $\bar \varphi_{2n}$, by Harnack inequality, one gets either $\bar\varphi_{2n}\rightarrow -\infty$ locally uniformly in $\mathbb R^2$ or $\bar\varphi_{2n}\rightarrow \bar \varphi_2$ in $W^{2,\gamma}_{loc}(\mathbb R^2)$. Following the same arguments as in Lemma \ref{lembl1}, we can exclude the previous case. In fact, by the same arguments, we also have $\lambda_{2n}+2(\beta_2-1)\ln t_n\geq -C$, otherwise,
\begin{equation*}
\Delta \bar \varphi_1=0,\bar \varphi_1(x_0)=0, x_0=\lim_{n\rightarrow \infty}\frac{y_{1n}}{t_n}\Rightarrow \bar \varphi_1\equiv 0
\end{equation*}which contradicts to $\int_{\mathbb R^2}|x|^{2\beta_1-4}e^{\bar \varphi_1}dx<+\infty$ by Fatou's lemma. And $\bar \varphi_1,\bar \varphi_2$ satisfy
\begin{equation*}
\begin{cases}
&-\Delta \bar\varphi_1=c_0|x|^{2\beta_2-4}e^{\bar \varphi_2},\\
&-\Delta \bar \varphi_2=|x|^{2\beta_1-4}e^{\bar \varphi_1},
\end{cases}\text{ in }\mathbb R^2
\end{equation*}with $|x|^{2\beta_1-4}e^{\bar\varphi_1},|x|^{2\beta_2-4}e^{\bar\varphi_2}\in L^1(\mathbb R^2)$ for some constant $0<c_0\leq 1$. Set
\begin{equation*}
A_1=\frac1{2\pi}\int_{\mathbb R^2}c_0|x|^{2\beta_2-4}e^{\bar \varphi_2}dx,A_2=\frac1{2\pi}\int_{\mathbb R^2}|x|^{2\beta_1-4}e^{\bar \varphi_1}dx.
\end{equation*}By Pohozaev's identity and repeating the arguments in Lemma \ref{lemPI}, one gets
$$A_1A_2=2(\beta_2-1)A_1+2(\beta_1-1)A_2.$$Noting $L_1\geq A_1,L_2\geq A_2$, we have
\begin{equation}\label{infty1}
L_1L_2\geq 2(\beta_2-1)L_1+2(\beta_1-1)L_2.
\end{equation}
\item[Case 2.] $\frac{|y_{1n}|}{t_n}\rightarrow +\infty$. Recall the definition of $w_{in},i=1,2$ in \eqref{BL-4} and set
$$\bar w_{in}(x)=w_{in}(x/|y_{1n}|)-2\ln |y_{1n}|,i=1,2.$$
$\bar w_{in}$ should satisfy
\begin{equation*}
\begin{cases}
\Delta\bar  w_{1n}+e^{\bar w_{2n}}(1-r_n^2|y_{1n}|^2e^{\bar w_{1n}})=\displaystyle 4\pi\sum_{i=1}^{N_1}\delta_{r_n\epsilon_n|y_{1n}| p_{1i}},\\
\Delta\bar  w_{2n}+e^{\bar w_{1n}}(1-r_n^2|y_{1n}|^2e^{\bar w_{2n}})=\displaystyle 4\pi\sum_{i=1}^{N_2}\delta_{r_n\epsilon_n|y_{1n}| p_{2i}},
\end{cases}\text{ in }\mathbb R^2.
\end{equation*}
Notice that
\begin{equation*}
\begin{split}
\bar w_{1n}\left(\frac{y_{1n}}{|y_{1n}|}\right)
=\varphi_{1n}(y_{1n})+2(\beta_1-1)\ln |y_{1n}|
=2(\beta_1-1)\ln \frac{|y_{1n}|}{t_n}\rightarrow +\infty.
\end{split}
\end{equation*}
By Theorem \ref{thmBM1}, we know that $\bar w_{2n}(x)$ also blows up at $q^*=\lim_{n\rightarrow \infty}\frac{y_{1n}}{|y_{1n}|}\in \mathbb S^1$.
This implies that along a subsequence, $\bar w_{in}$ has a finite non-empty blow up set $S^*$ and the elements in $S^*$ are non-zero points.
\end{itemize}
Turning back to $\varphi_{in},i=1,2$ and following the same arguments as in Lemma \ref{lemorigin}, we may assume
\begin{equation*}
\begin{split}
& |x|^{2\beta_2-4}e^{\varphi_{2n}}(1-r_n^2|x|^{2\beta_1}e^{\varphi_{1n}})\rightarrow 2\pi L_1\delta_0+2\pi M\sum_{q\in\tilde S}\delta_q,\\
&|x|^{2\beta_1-4}e^{\varphi_{1n}}(1-r_n^2|x|^{2\beta_2}e^{\varphi_{2n}})\rightarrow 2\pi L_2\delta_0+2\pi N\sum_{q\in\tilde S}\delta_q
\end{split}
\end{equation*}
and  also
\begin{equation*}
\begin{split}
\varphi_{1n}(x)-c_{1n}\rightarrow -L_1\ln|x|-M\sum_{q\in \tilde S}\ln |x-q|,\varphi_{2n}(x)-c_{2n}\rightarrow -L_2\ln|x|-N\sum_{q\in \tilde S}\ln |x-q|
\end{split}
\end{equation*}
in $C^1_{loc}(\mathbb R^2\backslash(\tilde S\cup\{0\}))$ for some constants $c_{1n},c_{2n}\rightarrow -\infty$. Also we have the following Pohozaev's identity
\begin{equation*}
\begin{split}
&\int_{\partial \Omega}(x\cdot\nabla \varphi_{2n})(\nu\cdot\nabla \varphi_{1n})dS+\int_{\partial \Omega}(x\cdot\nabla \varphi_{1n})(\nu\cdot\nabla \varphi_{2n})dS-\int_{\partial \Omega}(x\cdot\nu )(\nabla \varphi_{2n}\cdot\nabla \varphi_{1n})dS+o(1)\\
=&2(\beta_2-1)\int_{\Omega}|x|^{2\beta_2-4}e^{\varphi_{2n}}dx+2(\beta_1-1)\int_{\Omega}|x|^{2\beta_1-4}e^{\varphi_{1n}}dx
-2(\beta_1+\beta_2-1)\times\\&
\int_{\Omega}r_n^2|x|^{2\beta_1+2\beta_2-4}e^{\varphi_{1n}+\varphi_{2n}}dx.
\end{split}
\end{equation*}
Repeating the arguments of Lemma \ref{lemorigin} in $\Omega=\cup_{q\in\tilde S}B_r(q)$, we obtain that
\begin{equation}\label{infty3}
MN(|S|+1)+NL_1+ML_2=2\beta_1N+2\beta_2M.
\end{equation}
If Case 1 happens, by \eqref{infty1}, we have $L_i\geq 2(\beta_i-1)$ which implies
\begin{equation*}
\begin{split}
MN(|S|+1)+NL_1+ML_2\geq 2\beta_1 N+2\beta_2 M+(MN-2M-2N).
\end{split}
\end{equation*}
This yields a contradiction to \eqref{infty3} as $MN-2M-2N>0$. Hence, we must have Case 2 happens. Then by the proof of Lemma \ref{lemorigin}, we have \begin{equation*}
o(1)=\int_{|x|\leq d}e^{\bar w_{2n}}(1-r_n^2|y_{1n}|^2e^{\bar w_{1n}})dx=\int_{|x|\leq d/|y_{1n}|}e^{w_{2n}}(1-r_n^2e^{w_{1n}})dx\geq 2\pi|S|M
\end{equation*}for some small $d>0$ which yields a contradiction. This implies that $L_1=L_2=0$
\begin{equation}\label{infty4}
MN(|S|+1)=2\beta_1N+2\beta_2M
\end{equation}and $\varphi_{in}\rightarrow -\infty$ uniformly in $B_s(0)$.
\end{proof}
Now Lemma \ref{lemorigin} and Lemma \ref{leminfty} tell us that there is no concentration of mass at 0 and $\infty$ for $w_{in},i=1,2.$ This implies that \begin{equation}
\label{totalmass}
|S|M=2(\beta_1+N_1),|S|N=2(\beta_2+N_2).
\end{equation}
Combining \eqref{origin2}, \eqref{infty4} and \eqref{totalmass} together, we can get
\begin{equation}
|S|=\frac{(\beta_1+N_1)(\beta_2+N_2)}{\beta_1\beta_2-N_1N_2}.
\end{equation}
This implies that
\begin{equation}
\label{number}
\frac{N_1}{\beta_1+N_1}+\frac{N_2}{\beta_2+N_2}=\frac{|S|-1}{|S|}\in \{\frac{k-1}{k},k=2,\cdots,\max(N_1,N_2)\}.
\end{equation}
From \eqref{CS-3}, we know that $\frac{N_1+1}{\beta_1+N_1}+\frac{N_2+1}{\beta_2+N_2}< 1$. Combining with \eqref{number}, we get $\frac 1{\beta_1+N_1}+\frac{1}{\beta_2+N_2}<\frac 1{k}$. Then we have
\begin{equation*}
\frac{k-1}{k}\leq \frac{\max(N_1,N_2)}{\beta_1+N_1}+\frac{\max(N_1,N_2)}{\beta_2+N_2}<\frac{\max(N_1,N_2)}{k}
\end{equation*}which implies $k\leq \max(N_1,N_2)$. Theorem \ref{mainthm2} follows from \eqref{number} immediately.

\section{The existence of non-topological solutions}

 Set $v_{i\epsilon}=u_{i\epsilon}-h_{i\epsilon}$, where $\displaystyle h_{i\epsilon}=2\sum_{j=1}^{N_i}\ln|x-\epsilon p_{ij}|-(N_i+\beta_i)\ln(1+|x|^2)$, $i=1,2$.
Then $v_{i\epsilon}$ satisfies that
\begin{equation}\label{Global1}
\begin{split}
 \Delta v_{1\epsilon}+e^{v_{2\epsilon}+h_{2\epsilon}}(1-e^{v_{1\epsilon}+h_{1\epsilon}})=\frac{4(N_1+\beta_1)}{(1+|x|^2)^2}=g_{1},\\
 \Delta v_{2\epsilon}+e^{v_{1\epsilon}+h_{1\epsilon}}(1-e^{v_{2\epsilon}+h_{2\epsilon}})=\frac{4(N_2+\beta_2)}{(1+|x|^2)^2}=g_{2}.
\end{split}
\end{equation}
\begin{theorem}
\label{mainthm3}Under the assumption of Theorem \ref{mainthm2}, we have
$$|v_{1\epsilon}|_{L^\infty(\mathbb R^2)}+|v_{2\epsilon}|_{L^\infty(\mathbb R^2)}\leq C $$ for some constant $C$ independent of $\epsilon$.
\end{theorem}
\begin{proof}
By Theorem \ref{mainthm2}, we have $v_{i\epsilon}$ is bounded in $L^\infty_{loc}(\mathbb R^2)$. Consider the function
\begin{equation*}
\xi_{i\epsilon}(x)=u_{i\epsilon}\left(\frac x{|x|^2}\right)-2\beta_i\ln |x|,\text{ for }|x|\leq R_0:=(1+\max_{i,j}|p_{ij}|)^{-1}.
\end{equation*}
Obviously, $\xi_{i\epsilon}(x)\in L^\infty_{loc}(B_{R_0}\backslash \{0\})$ and
\begin{equation}\label{Global2}
\begin{split}
\Delta \xi_{1\epsilon}+|x|^{2\beta_2-4}e^{\xi_{2\epsilon}}(1-|x|^{2\beta_1}e^{\xi_{1\epsilon}})=0,\\
\Delta \xi_{2\epsilon}+|x|^{2\beta_1-4}e^{\xi_{1\epsilon}}(1-|x|^{2\beta_2}e^{\xi_{2\epsilon}})=0,
\end{split}\quad\quad\text{ in } |x|\leq R_0.
\end{equation}
We now show $\xi_{i\epsilon},i=1,2$ is bounded from above in $B_{R_0}(0)$. If not, we may assume a sequence of $\lambda_{1n}=\xi_{1n}(y_{1n})=\max_{|x|\leq R_0}\xi_{1n}\rightarrow +\infty$. In fact, we must have $\lambda_{2n}=\xi_{2n}(y_{2n})=\max_{|x|\leq R_0}\xi_{2n}\rightarrow +\infty$. Otherwise, by Harnack inequality, we will have $\xi_{1n}\rightarrow +\infty$ in $B_{R_0}$ which contradicts to $\xi_{1n}(x)\in L^\infty(\partial B_{R_0})$.
 \par We claim that
 $$\max(\xi_{1n}(y_{1n})+2(\beta_1-1)\ln |y_{1n}|,\xi_{2n}(y_{2n})+2(\beta_2-1)\ln |y_{2n}|)\rightarrow +\infty.$$
 Otherwise, without loss of generality, we can choose $t_n$ such that
 $$\xi_{1n}(y_{1n})+2(\beta_1-1)\ln t_n=0, \xi_{2n}(y_{2n})+2(\beta_2-1)\ln t_n\leq 0.$$
 Consider the scaled function $\zeta_{in}(x)=\xi_{in}(t_n x)-\xi_{in}(y_{in})$.  Set
 \begin{equation*}
 L_1=\lim_{n\rightarrow \infty} \int_{|x|\leq R_0}|x|^{2\beta_2-4}e^{\xi_{2n}}(1-|x|^{2\beta_1}e^{\xi_{1n}})dx, L_2=\lim_{n\rightarrow \infty}\int_{|x|\leq R_0}|x|^{2\beta_1-4}e^{\xi_{1n}}(1-|x|^{2\beta_2}e^{\xi_{2n}})dx.
 \end{equation*}
 Repeating the same arguments as in Lemma \ref{leminfty}, we must have
 $$\frac{2(\beta_1-1)}{L_1}+\frac{2(\beta_2-1)}{L_2}\leq 1. $$
 As $L_1\leq 2(\beta_1+N_1),L_2\leq 2(\beta_2+N_2)$, substituting this to the above inequality yields that
 $$(\beta_1-1)(\beta_2-1)\leq (N_1+1)(N_2+1)$$ which is a contradiction to \eqref{CS-3}. This proves the claim.
 \par We may assume $\xi_{1n}(y_{1n})+2(\beta_1-1)\ln |y_{1n}|\rightarrow +\infty$. Set $\bar u_{in}(x)=u_{in}(x/|y_{1n}|)-2\ln |y_{1n}|$.
 Noting $\bar u_{1n}(y_{1n}/|y_{1n}|)\rightarrow +\infty$ and the equations $\bar u_{in}$ satisfying, along a subsequence, $\bar u_{in}$ satisfies the blow up situation in Theorem \ref{thmBM1} and the blow up set $\bar S$ contains at least one non-zero point on $\mathbb S^1$. Then the proof of Lemma \ref{lemorigin} implies that
 $$o(1)=\int_{|x|\leq d} e^{\bar u_{1n}}(1-|y_{1n}|^2 e^{\bar u_{2n}})dx=\int_{|y|\leq d/|y_{1n}|}e^{u_{1n}}(1-e^{u_{2n}})dx.$$
 Theorem \ref{mainthm2} tells us that $u_{in}$ is uniformly bounded in $C^2_{loc}(\mathbb R^2\backslash B_R(0))$, $R>\max(|p_{ij}|+1)$. From $\int_{\mathbb R^2} e^{u_{2n}}dx\leq C<\infty$, one gets $|\{x|u_{2n}(x)>-1\}|\leq C_1$. Now we have
  \begin{equation*}
  \int_{R\leq|y|\leq C_2 R} e^{u_{1n}} (1-e^{u_{2n}})dx\geq c(1-e^{-1})(C_2 R^2-C_1)\geq C_0>0
  \end{equation*}
  for some fixed $C_2$ large enough. This yields a contradiction.
 \par  By now, we have proved that $\xi_{i}$ is bounded from above. Then if we note that $\xi_{i}\in L^\infty(\partial B_{R_0})$, by standard elliptic estimates, we can prove $\xi_{i}$ is uniformly bounded from below too.
\end{proof}
Due to Theorem \ref{mainthm3}, we now can calculate Leray-Schauder degree for \eqref{CS-5}. Recall the following Hilbert space defined in Section 2 for $\beta=\min(\beta_1,\beta_2,2)>1$
$$\mathcal{D}=\{v:\mathbb R^2\rightarrow \mathbb R\ |\   |v|_{\mathcal{D}}^2=\int_{\mathbb R^2}|\nabla v|^2 dx+\int_{\mathbb R^2}\frac{v^2}{(1+|x|^2)^{\beta}}dx<+\infty\}.$$
For every $v\in \mathcal D$, there holds that for any $\beta>1$
\begin{equation}\label{embed}
\ln \int_{\mathbb R^2}\frac{e^v}{(1+|x|^2)^\beta}dx\leq \frac{1}{8\pi \gamma}|\nabla v|_{L^2}^2+\bar v +C_{\gamma}
\end{equation}
for any positive $\gamma<2(\beta-1)$ if $1<\beta<2$ and $\gamma=2$ if $\beta\geq 2$. Here,
$$\bar v=\frac 1{\pi} \int_{\mathbb R^2}\frac{v}{(1+|x|^2)^\beta}dx.$$
The above inequality can be found in \cite{Kim2006}.
\begin{lemma}\label{lemaa}
Under the assumption of Theorem \ref{mainthm2}, we have
$$|v_{1\epsilon}|_{\mathcal D}+|v_{2\epsilon}|_{\mathcal D}\leq C$$
for some constant $C$ independent of $\epsilon$.
\end{lemma}
The proof of Lemma \ref{lemaa} is simply integrating by parts and using Theorem \ref{mainthm3}. We omit the details here.
\par \textbf{The proof for Theorem \ref{mainthm1}:}
Now we can prove Theorem \ref{mainthm1} by Leray-Schauder degree theory. We now define a map as follows.
$$T(\epsilon,v_1,v_2)=(T_1(\epsilon,v_1,v_2),T_2(\epsilon,v_1,v_2)): \mathcal D\times\mathcal D\rightarrow \mathcal D\times\mathcal D$$
for $\epsilon\in [0,1]$ by
\begin{equation}\label{degreemap}
\begin{split}
& T_{1}(\epsilon,v_1,v_2)=(-\Delta+\sigma)^{-1}[e^{v_2+h_{2\epsilon}}(1-e^{v_1+h_{1\epsilon}})+\sigma v_1-g_1],\\
& T_{2}(\epsilon,v_1,v_2)=(-\Delta+\sigma)^{-1}[e^{v_1+h_{1\epsilon}}(1-e^{v_2+h_{2\epsilon}})+\sigma v_2-g_2]
\end{split}
\end{equation}where $\sigma=\frac {1}{(1+|x|^2)^\beta}$.
\par It is obvious that $T(\epsilon,v_1,v_2): \mathcal D\times\mathcal D\rightarrow \mathcal D\times \mathcal D$ is compact by \eqref{embed}. And by Lemma \ref{lemaa}, there exists a constant $R>0$ such that every zero of $I-T(\epsilon,v_1,v_2)$ is contained in a ball $\Omega_R=\{(v_1,v_2)\in \mathcal D\times\mathcal D\  |\  |v_1|_{\mathcal D}+|v_2|_{\mathcal D}<R\}$. Then the degree $\deg(I-T(\epsilon,v_1,v_2),\Omega_R,0)$ is well defined. Moreover, $I-T(\epsilon,v_1,v_2)$ is a continuous homotopy with respect to $\epsilon$ and preserves degree by Lemma \ref{lemaa}.
\par By the above arguments, we only need to calculate the Leray-Schauder degree of $I-T(0,v_1,v_2)$. It is well known that non-radial solutions of $I-T(0,v_1,v_2)=0$, if they exist, do not affect the calculation of $\deg(I-T(0,v_1,v_2),\Omega_R,0)$. See \cite{Wang1989} and references therein. Then by the proof of Theorem \ref{thmradial}, we have
$$\deg(I-T(1,v_1,v_2),\Omega_R,0)=\deg(I-T(0,v_1,v_2),\Omega_R,0)=-1.$$
This proves Theorem \ref{mainthm1}.

\section*{Acknowledgement}
The first author would like to thank Taida Institute for Mathematical Sciences (TIMS), National Taiwan University for the warm hospitality where this work was done.

\medskip
\medskip


\begin{thebibliography}{99}

\bibitem{BartolucciChenLinTarantello2004}
D. Bartolucci, C.-C. Chen, C.-S. Lin and G. Tarantello, Profile of blow-up solutions to mean field equation with sigualar data, Comm. Partial Differential Equations 29 (2004), pp. 1241-1265.

\bibitem{BrezisMerle1991}
H. Brezis and F. Merle, Uniform estimates and blow-up behavior for solutions of $-\Delta u=V(x)e^u$ in two dimensions. Comm. Partial Diff. Eq. 16 (8-9), pp. 1223-1253.

\bibitem{CaffarelliYang1995}
L. Caffarelli and Y. Yang, Vortex condensation in the Chern-Simons-Higgs model: an existence theorem, Comm. Math. Phys., 168 (1995), pp.321-336.

\bibitem{ChanFuLin2002}
H. Chan, C. Fu and C.-S. Lin, Non-topological multi-vortex solutions to the self-dual Chern-Simons-Higgs equation, Comm. Math. Phys., 231 (2002), pp. 1465-1507.

\bibitem{ChernChenLin2010}
J. Chern, Z. Chen and C.-S. Lin, Uniqueness of topological solutions and the structure of solutions for the Chern-Simons system with two Higgs particles, Comm. Math. Phys., 296 (2010), pp. 323-351.

\bibitem{ChenLin2002}
C.-C. Chen and C.-S. Lin, Sharp estimates for solutions of multi-bubbles in compact Riemann surfaces. Comm. Pure Appl. Math. 55 (2002), no. 6, 728-771.

\bibitem{Choe2007}
K. Choe, Asympototic behavior of condensate solutions in the Chern-Simons-Higgs theory, J.Math.Phys. 48 (2007) 103501.

\bibitem{ChoeKim2008}
K. Choe and N. Kim, Blow-up solutions of the self-dual Chern-Simons-Higgs vortex equation, Ann. Inst. H. Poincar$\acute{e}$ Anal. Non Lin$\acute{e}$aire 25 (2008), pp. 313-338.

\bibitem{ChoeKimLin2011}
K. Choe, N. Kim and C.-S. Lin, Existence of self-dual non-topological solutions in the Chern-Simons Higgs model,  Ann. Inst. H. Poincar$\acute{e}$ Anal. Non Lin$\acute{e}$aire 28 (2011), pp. 837-852.

\bibitem{Dunne1999}
G. V. Dunne, Aspects of Chern-Simoms theory, in Aspects topologiques de la physique en basse dimension/Topological aspects of low dimensional systems(Les Houches, 1998), EDP Sci., Les Ulis, 1999, pp. 177-263.

\bibitem{Dziarmaga1994}
J. Dziarmaga, Low energy dynamics of $[U(1)]^{N}$ Chern-Simons solitons, Phys. Rev. D 49 (1994), pp. 5469-5479.

\bibitem{HuangLin2013}
H.-Y. Huang and C.-S. Lin, Uniqueness of non-topological solutions for the Chern-Simons system with two Higgs particles, to appear in Kodai J.


\bibitem{JaffeTaubes1980}
A. Jaffe and C. Taubes, Vortices and monopoles, vol. 2 of Progress in Physics, Birkh\"auser Boston, Mass., 1980 Structure of static gauge theories.

\bibitem{Kim2006}
N. Kim, Existence of vortices in a self-dual gauged linear sigma model and its singular limit, Nonlinearity, 19 (2006), pp. 721-739.

\bibitem{KimLeeKoLeeMin1993}
C. Kim, C. Lee, P. Ko, B.-H. Lee and H. Min, Schr\"odinger fields on the plane with $[U(1)]^N$ Chern-Simons interactions and generalized self-dual solitons, Phys. Rev. D (3) 48 (1993), pp. 1821-1840.

\bibitem{LiShafrir1994}
Y. Li and I. Shafrir, Blowup analysis for solutions $-\Delta u=Ve^u$ in dimension two, Indiana Univ. Math. Jour. 43 (1994), pp. 1255-1270.

\bibitem{LinPonceYang2007}
C.-S. Lin, A. Ponce and Y. Yang, A system of elliptic equations arising in Chern-Simons field theory, Journal of Functional Analysis, 247 (2007), pp.289-350.



\bibitem{LinYan2013}
C.-S. Lin and S. Yan, Existence of bubbling solutions for Chern-Simons model on a torus, Arch. Ration. Mech. Anal., 207 (2013), no. 2, pp. 353-392.

\bibitem{Nirenberg2001}
L. Nirenberg, Topics in Nonlinear Functional Analysis, Courant Lecture Notes in Mathematics, American Mathematics Society, Rhode Island, 2001.

\bibitem{NolascoTarantello2000}
M. Nolasco and G. Tarantello, Vortex condensates for the $SU(3)$ Chern-Simons Theory, Comm. Math. Phy. 213 (2000), no. 3, pp. 599-639.

\bibitem{SpruckYang1992}
J. Spruck and Y. Yang, The exitence of nontopological solutions in the self-dual Chern-Simons theory, Comm. Math. Phys., 149 (1992), pp. 361-376.


\bibitem{Tarantello1996}
G. Tarantello, Multiple condensate solutions for the Chern-Simons-Higgs theory, J. Math. Phys., 37 (1996), pp. 3769-3796.

\bibitem{Hooft1979}
G. 't Hooft, A property of electric and magnetic flux in non-Abelian gauge theories, Nuclear Phys. B 153 (1979), no.1-2, pp. 141-160.


\bibitem{Wang1989}
Z. Wang, Symmetries and the calculations of degree, Chin. Ann. of Math. B 16 (1989), pp. 520-536.

\bibitem{Yanagida2001}
E. Yanagida,
Reaction-diffusion systems with skew-gradient structure,
Methods and Applications of Analysis,
8 (2001), pp. 209-226.

\bibitem{Yanagida2002}
E. Yanagida, Mini-maximizers in reaction-diffusion systems
with skew-gradient structure, J. Diff. Eqs., 179 (2002),
pp. 311-335.

\bibitem{Yang1997}
Y. Yang, the relativistic non-abelian Chern-Simons equations, Comm. Math. Phys. 186(1997), no.1,  pp.141-160.

\end{thebibliography}
\end{document}